\documentclass[review]{elsarticle}

\usepackage{lineno,hyperref}

\usepackage{amsmath, amsthm, amsfonts}
\usepackage{hyperref}
\usepackage{float}
\usepackage{graphicx}
\usepackage{multicol}
\usepackage{xcolor}

\newtheorem{thm}{Theorem}[section]
\newtheorem{theorem}[thm]{Theorem}
\newtheorem{corollary}[thm]{Corollary}
\newtheorem{definition}[thm]{Definition}
\newtheorem{lemma}[thm]{Lemma}
\newtheorem{proposition}[thm]{Proposition}
\newtheorem{example}[thm]{Example}
\theoremstyle{definition}
\theoremstyle{remark}

\journal{}

\begin{document}

\begin{frontmatter}

\title{On multiplicative Jacobi polynomials and function approximation through multiplicative series}
%\tnotetext[mytitlenote]{Fully documented templates are available in the elsarticle package on \href{http://www.ctan.org/tex-archive/macros/latex/contrib/elsarticle}{CTAN}.}

%% Group authors per affiliation:
\author[mymainaddress]{\emph{Edinson Fuentes}}
\address{Facultad de Ciencias B\'asicas e Ingenier\'ia, Universidad de los Llanos}
\ead{edfuentes@unillanos.edu.co}

\author[mysecondaryaddress]{{\normalsize Luis E. Garza}\fnref{myfootnote}}
\address{Facultad de Ciencias, Universidad de Colima}
\ead{luis\_garza1@ucol.mx}

\author[mymainaddress]{{\normalsize Fabian Vel\'azquez C.}}
\address{Facultad de Ciencias B\'asicas e Ingenier\'ia, Universidad de los Llanos}
\ead{fvelasquez@unillanos.edu.co}

\fntext[myfootnote]{Corresponding author}
%% or include affiliations in footnotes:
%\author[mymainaddress]{Edinson Fuentes}
%\ead{efuentes@unal.edu.co}
%
%\author[mysecondaryaddress]{Luis E. Garza}
%\cortext[mycorrespondingauthor]{Corresponding author}
%\ead{luis\_garza1@ucol.mx}

\address[mymainaddress]{Km. 12 V\'ia a Puerto L\'opez, Vda. Barcelona, Villavicencio, Meta, Colombia}
\address[mysecondaryaddress]{Bernal D\'iaz del Castillo 340, Colima, M\'exico}

\begin{abstract}
In this contribution, we introduce the multiplicative Jacobi polynomials that arise as one of the solutions of the  multiplicative Sturm-Liouville equation
\begin{equation*}
    \frac{d^*}{dx}\left( e^{(1-x^2)\omega(x)}\odot \frac{d^*y}{dx} \right)\oplus \left(e^{  n(n+\alpha+\beta+1)\omega(x)}\odot y\right)=1, \ x\in[-1,1],
\end{equation*}
where $\omega(x)=(1-x)^{\alpha}(1+x)^{\beta}$ with $\alpha, \beta >-1$ real numbers and $n$ is a non-negative integer number. We extend some properties of classical Jacobi polynomials to the multiplicative case. In particular, we present several properties of multiplicative Legendre polynomials and multiplicative Chebyshev polynomials of first and second kind. We also prove that every real and positive function can be expressed as a multiplicative Jacobi-Fourier series and show that such functions can be approximated by the corresponding partial products of these series. We illustrate the obtained results with some examples.
\end{abstract}

\begin{keyword}
Multiplicative Jacobi polynomials \sep function approximation \sep multiplicative differential equations.
\MSC[2010]  42C05 \sep 33C45 \sep 65D15 \sep 34A99.
\end{keyword}

\end{frontmatter}

%\linenumbers

\section{Introduction}

Classical differential calculus, also known as Newtonian calculus, was originated in the $17$th century from the classical definition of the derivative of a function $f$ given by 
\begin{equation*}
    f'(x):=\frac{d}{dx}f(x)=\lim_{h\to 0}\frac{f(x+h)-f(x)}{h},
\end{equation*}
as long as the limit exists. Alternative definitions have been proposed since then, given rise to other branches of calculus different to the Newtonian calculus, commonly known as non-Newtonian calculus. One of these alternative definitions is the multiplicative derivative introduced by Grossman y Katz \cite{Grossman}, defined as follows.
\begin{definition}
Let $\mathbb{R}_{e}:=\{e^x:x\in\mathbb{R}\}$ be the set of exponential real numbers and let $f:\mathbf{A} \subseteq \mathbb{R}  \rightarrow \mathbb{R}_{e}$ be a function. The multiplicative derivative or $^*$-derivative of $f$ is defined by
    \begin{equation*}%\label{Jac-Mul-23}
        f^*(x):=\frac{d^*}{dx}f(x)=\lim_{h \to 0} \left [\frac{f(x+h)}{f(x)} \right]^{1/h},
    \end{equation*} 
if the above limit exists and it is positive.
\end{definition}
%Sin embargo, han surgido otros definiciones de derivadas, permiten abordar algunos problemas con más precisión, como lo es la multiplicative derivative (see \cite{Bashirov-2}) given by

The above definition has helped consolidate what is known as multiplicative calculus \cite{Bashirov-2,Stanley,Uzer}. Although one of its main limitations is that it is only applicable to positive functions, multiplicative calculus has proven to be a more efficient tool for addressing certain problems that involve exponentially varying functions or functions of exponential order. Many natural phenomena, such as population growth \cite{20} or heat transfer \cite{Sivasankaran}, exhibit this type of behavior.

In this context, references \cite{18,19,20} present a new method for solving some nonlinear equations using a multiplicative calculus approach and apply the results to phenomena modeled by exponential growth functions. Furthermore, they demonstrate numerically that their results are significantly better than those obtained with traditional methods. Conversely, \cite{Uzer} shows that better results are achieved by using the multiplicative derivative when tables and graphs are expressed in logarithmic scales. Additionally, \cite{Yalcin-3} applies multiplicative calculus in numerical methods for solving multiplicative partial differential equations, while classical methods for solving differential equations, such as the Runge-Kutta method, are discussed in \cite{Riza}. Moreover, there are applications in the analysis of biomedical images \cite{Florack}. As a consequence, multiplicative calculus has emerged as a growing field that offers alternative solutions to classical problems. 

On the other hand, orthogonal polynomials constitute a well-known example of the so-called special functions. These polynomials have been widely used in different fields such as differential equations in Sturm-Liouville problems (from a classical calculus standpoint), approximation theory, numerical analysis, quantum physics, among others (we refer the reader to \cite{Boy,Burden,Chihara,Ismail,Szego} for more details). One of the orthogonal families that stands out for its many applications, specially in approximation processes, are the Jacobi Polynomials $\{P_n^{(\alpha,\beta)}\}_{n\geq 0}$ with $\alpha,\beta >-1$. They appear as the polynomial solutions for the Jacobi differential equation, which is a Sturm-Liouville equation of the form (see \cite{Chihara})
\begin{equation*}
\frac{d}{dx}\left[ (1-x^2)\omega(x)\frac{dy}{dx} \right]+n(n+\alpha+\beta+1)\omega(x)y=0, \ \ x\in[-1,1],
\end{equation*}
where $\omega(x)=(1-x)^{\alpha}(1+x)^{\beta}$ is a weight function and $n$ is a non-negative integer number. Every Jacobi polynomial has the important representation (see \cite{Szego})
\begin{equation}\label{Jac-Mul-26}
    P_n^{(\alpha,\beta)}(x)=\frac{1}{n!}\left[(\alpha+1)_n+\sum_{k=1}^{n}\binom{n}{k}(\alpha+\beta+n+1)_{k}\frac{(\alpha+1)_n}{(\alpha+1)_k}\left( \frac{x-1}{2} \right)^k\right], \ \ n\geq 1,
\end{equation}
where $(\alpha)_n$ is the Pochhammer symbol or shifted factorial, defined by 
\begin{equation}\label{Jac-Mul-27}
  (\alpha)_0:=1, \ \ (\alpha)_n:=\alpha(\alpha+1)\cdots(\alpha+n-1), \ \ n> 0.
\end{equation}
%with $\alpha +n$ non-zero for all integer non-negative $n$. 
They are orthogonal with respect the weight function $\omega$ on the interval $[-1,1]$. That is, 
\begin{equation}\label{Jac-Mul-30}
    \int_{-1}^1 P_n^{(\alpha,\beta)}(x)P_m^{(\alpha,\beta)}(x)\omega(x)dx=\frac{2^{\alpha+\beta+1}}{2n+\alpha+\beta+1}\frac{\Gamma(n+\alpha+1)\Gamma(n+\beta+1)}{n!\Gamma(n+\alpha+\beta+1)}\delta_{n,m},
\end{equation}
where $\Gamma(z)=\int_{0}^{\infty}x^{z-1}e^{-x}dx$ is the Gamma function defined for every complex $z$ with positive real part and $\delta_{n,m}$ is the Kronecker delta. It is well-known that the Gamma function satisfies 
\begin{equation*}
    \frac{\Gamma(\alpha+n)}{\Gamma(\alpha)}=\alpha(\alpha+1)\cdots(\alpha+n-1)=(\alpha)_n.
\end{equation*}
%As a consequence the first Jacobi polynomials are 
%\begin{equation*}
%\begin{split}
%P_0^{(\alpha,\beta)}(x)&=1,\\
%P_1^{(\alpha,\beta)}(x)&=(\alpha+1)+(\alpha+\beta+2)\left(\frac{x-1}{2} %\right),\\
%P_2^{(\alpha,\beta)}(x)&=\frac{1}{2}\left[(\alpha+1)(\alpha+2)+2(\alpha+\beta+3)(\alpha+2)\left(\frac{x-1}{2} \right)+(\alpha+\beta+3)(\alpha+\beta+4)\left(\frac{x-1}{2}\right)^2\right].
%\end{split}
%\end{equation*}
Further properties of classical Jacobi polynomials are the following \cite{Chihara}.
\begin{proposition}\label{Jac-Mul-31}
Let $\{P_n^{(\alpha,\beta)} \}_{n\geq 0}$ be the sequence of Jacobi polynomials and $\omega(x)=(1-x)^{\alpha}(1+x)^{\beta}$ ($\alpha,\ \beta >-1$). Then,
    \begin{enumerate}
        \item $\frac{d}{dx}\left(P_n^{(\alpha,\beta)}(x)\right)=\frac{1}{2}(n+\alpha+\beta+1)P_{n-1}^{(\alpha+1,\beta+1)}(x)$.
        \item Rodrigues' formula
\begin{equation*}
    P_n^{(\alpha,\beta)}(x)=\frac{1}{(-1)^n n!2^n\omega(x)}\frac{d^n}{dx^n}\left[(1-x^2)^n\omega(x)\right], \ \ n\geq 0.
\end{equation*}
        \item Three term recurrence relation of the form
\begin{equation*}
    A(n,\alpha, \beta)P_{n}^{(\alpha,\beta)}(x)=B(x,n,\alpha, \beta)P_{n-1}^{(\alpha,\beta)}(x)-C(n,\alpha, \beta)P_{n-2}^{(\alpha,\beta)}(x), \ \ n\geq 1,
\end{equation*}
with $P_{0}^{(\alpha,\beta)}(x)=1$ and $P_{-1}^{(\alpha,\beta)}(x)=0$, where
\begin{equation}\label{Jac-Mul-32}
\begin{split}
    A(n,\alpha, \beta)&=2n(n+\alpha+\beta)(2n+\alpha+\beta-2),\\
    B(x,n,\alpha, \beta)&=(2n+\alpha+\beta-1)\left[(2n+\alpha+\beta)(2n+\alpha+\beta-2)x+\alpha^2-\beta^2 \right],\\
    C(n,\alpha, \beta)&=2(n+\alpha-1)(n+\beta-1)(2n+\alpha+\beta).
\end{split}    
\end{equation}
\item $P_n^{(\alpha,\beta)}(1)=\binom{n+\alpha}{n}$ and $P_n^{(\alpha,\beta)}(-x)=(-1)^nP_n^{(\beta,\alpha)}(x)$. The symbol $\binom{\cdot}{\cdot}$ represents the generalized binomial coefficient and $\binom{n+\alpha}{n}=\frac{\Gamma(n+\alpha+1)}{\Gamma(n+1)\Gamma(\alpha+1)}$.
\item $P_n^{(\alpha,\beta)}(x)$ is a polynomial of degree $n$ whose leading coefficient is 
\begin{equation*}
    k_n\equiv k_n(\alpha,\beta)=\frac{1}{2^n}\binom{2n+\alpha+\beta}{n}, \ \ n\geq 0.
\end{equation*}
    \end{enumerate}
\end{proposition}

Moreover, the sequence $\{P_n^{(\alpha,\beta)} \}_{n\geq 0}$ of Jacobi polynomials constitutes a countable, dense and complete (total) basis (see \cite{Ismail,Szego}) for the weighted Hilbert space
\begin{equation*}
    \mathbf{L}^2\left([-1,1],\omega\right):=\left\{ f:[-1,1] \rightarrow \mathbb{R}  \ : \ \langle f,f\rangle_{\omega}= \int_{-1}^1  (f(x))^2{\omega(x)dx} < \infty  \right \},
\end{equation*}
with norm $\|f\|_{\omega}^2= \langle f,f\rangle_{\omega}$. As a consequence, any $f\in \mathbf{L}^2\left([-1,1],\omega\right)$ can be uniquely expressed as
\begin{equation*}
    f(x)=\sum_{n=0}^{\infty}\frac{\langle f, P_n^{(\alpha,\beta)} \rangle_{\omega}}{\|P_n^{(\alpha,\beta)} \|^2_{\omega}}P_n^{(\alpha,\beta)}(x), \ \ x\in [-1,1],
\end{equation*}
where the Jacobi-fourier series converges to $f$ in the norm $\mathbf{L}^2\left([-1,1],\omega\right)$, i.e. 
\begin{equation}\label{Jac-Mul-56}
    \lim_{N\to \infty}\left\|f-\sum_{n=0}^{N}\frac{\langle f, P_n^{(\alpha,\beta)} \rangle_{\omega}}{\|P_n^{(\alpha,\beta)} \|^2_{\omega}}P_n^{(\alpha,\beta)}(x)\right\|_{\omega}=0.
\end{equation}

On the other hand, several contributions on orthogonal polynomials, within the framework of multiplicative calculus, have recently appeared in the literature. For instance, multiplicative differential equations solved using multiplicative power series give rise to multiplicative orthogonal polynomials, and some spectral properties are studied. Examples of such families include: multiplicative Laguerre polynomials \cite{Kosunalp}, multiplicative Hermite polynomials \cite{Yalcin-2}, and multiplicative Chebyshev polynomials of the first and second kind \cite{Yalcin}, as well as multiplicative Legendre polynomials \cite{Goktas}. It is well-known that Chebyshev and Legendre polynomials are particular cases of Jacobi polynomials (see \cite{Chihara}), and, as far as we know, properties of multiplicative Jacobi polynomials have not been studied from the multiplicative calculus standpoint. Thus, the aim of this contribution is to introduce multiplicative Jacobi polynomials, study some of their properties, and explore their application in approximating positive functions from the multiplicative calculus standpoint.

The structure of the manuscript is as follows: Section $2$ contains properties related to both the multiplicative derivative and integral. Moreover, we present some properties of multiplicative series and other definitions needed in the sequel. In Section $3$, we introduce multiplicative Jacobi polynomials as finite multiplicative series solutions of the multiplicative Jacobi equation, which is a particular multiplicative Sturm-Liouville equation. Section $4$ deals with the properties of multiplicative Jacobi polynomials, including some properties of multiplicative Legendre polynomials and multiplicative Chebyshev polynomials of the first and second kind. Finally, in the last section, we prove that every positive function can be expressed as a multiplicative Jacobi-Fourier series and study its convergence properties. We also present illustrative examples where the use of multiplicative Jacobi-Fourier series for approximating functions is more efficient than using classical Jacobi-Fourier series, which motivates the study of properties of multiplicative series.

%===========================================================================================================================
\section{Multiplicative calculus and its properties }
%En esta sección, primero se presentan algunas propiedades both multiplicative derivative and integral. Moreover, se exponen algunas propiedades de multiplicative power series y otras definiciones utiles en el resto del documento.  

The multiplicative calculus has the following properties, that are similar to those of the classical calculus. The proofs can be found in \cite{Bashirov-2,Stanley,Uzer}.
\begin{proposition}\label{Jac-Mul-36}
(see \cite{Bashirov-2}) Let $f,g$ be multiplicative differentiable and $\phi$ be classical differentiable at $x$. The multiplicative derivative verifies the following properties: 
    \begin{enumerate}
\item $(k f)^*(x)=f^*(x)$, $k \in \mathbb{R}_{e}$,
\item $(fg)^*(x)=f^*(x)g^*(x)$,
\item $(f/g)^*(x)=f^*(x)/g^*(x)$,
\item $(f^{\phi})^*(x)=f^*(x)^{\phi(x)}f(x)^{\phi'(x)}$. In particular, if $f$ is a constant function 
\begin{equation}\label{Jac-Mul-33}
    (f^{\phi})^{*(n)}(x)=f(x)^{\phi^{(n)}(x)}, \ \ n\geq 0.
\end{equation}
\item $(f\circ \phi)^*(x)=(f^*\circ \phi)(x)^{\phi'(x)}$,
\item $(f+g)^*(x)=f^*(x)^{\frac{f(x)}{f(x)+g(x)}}g^*(x)^{\frac{g(x)}{f(x)+g(x)}}$.
    \end{enumerate}
\end{proposition}

In particular, if $f$ is positive definite and the classical derivative exists, there exists a relationship between the classical derivative and the multiplicative derivative (see \cite{Stanley})
\begin{equation}\label{Jac-Mul-23-b}
    f^*(x):=e^{ (\ln \circ f)'(x)}=e^{\frac{f'(x)}{f(x)}},
\end{equation}
where $(\ln \circ f)(x)=\ln(f(x))$. Moreover, the $n$-th multiplicative derivative of the positive function $f$ is defined by  $f^{*(n)}(x):=e^{ (\ln \circ f)^{(n)}(x)}$ (see \cite{Bashirov}).

\begin{proposition}\label{Jac-Mult-24}
    The function $f$ is a positive constant on $ (a,b)$ if and only if $f^*(x)=1$ for all $x\in (a,b)$.
\end{proposition}
\begin{proof}
    If $f(x)=c>0$ for all $x\in(a,b)$, the result follows at once from \eqref{Jac-Mul-23-b}. Conversely, if $f^*(x)=1$ for all $x\in (a,b)$, again from \eqref{Jac-Mul-23-b} we have $f^*(x)=e^{\frac{f'(x)}{f(x)}}=1$, so $f'(x)=0$ and the required result is immediate from the mean value theorem of classical calculus.   
\end{proof}

An analog of the Riemann integral for the multiplicative calculus was defined in \cite{Bashirov-2}. Moreover, the authors prove the following result that relates the multiplicative integral with the Riemann integral. 
\begin{proposition}
\cite{Bashirov-2} Sea $f:[a,b]\rightarrow \mathbb{R}_e$. If $f$ is positive and Riemann integrable on $[a,b]$, then $f$ is multiplicative integrable or $^*$-integrable and 
\begin{equation}\label{Jac-Mul-29}
    \int_{a}^bf(x)^{dx}=\exp \left( \int_{a}^b(\ln f(x))dx \right)=e^{\int_{a}^b(\ln f(x))dx}.
\end{equation}    
Conversely, if $f$ is $^*$-integrable on $[a,b]$, then 
$   \int_a^b f(x)dx=\ln \int_a^b \left( e^{f(x)}\right)^{dx}.$
\end{proposition}
Properties of the multiplicative integral that are analog to those of the Riemann integral are the following. 
\begin{proposition} 
\cite{Bashirov-2} Let $f,g:[a,b]\rightarrow \mathbb{R}_e$ be bounded and $^*$-integrable functions and $\phi:[a,b]\rightarrow \mathbb{R}_e$ be usual differentiable at $x\in[a,b]$. The $^*$-integral verifies the following properties:  
    \begin{enumerate}
        \item $\int_{a}^b\left(f(x)^k\right)^{dx}=\left(\int_{a}^bf(x)^{dx}\right)^k$, \ \ $k\in \mathbb{R}$,
        \item $\int_{a}^b\left(f(x)g(x)\right)^{dx}=\left(\int_{a}^bf(x)^{dx} \right)\left(\int_{a}^bg(x)^{dx}\right)$,
        \item $\int_{a}^b\left(\frac{f(x)}{g(x)}\right)^{dx}=\frac{\int_{a}^bf(x)^{dx}}{\int_{a}^bg(x)^{dx}}$,     
        \item Multiplicative integration by parts formula:
$$\int_{a}^b\left(f^*(x)^{\phi(x)}\right)^{dx}=\frac{f(b)^{\phi(b)}}{f(a)^{\phi(a)}}\left[\int_{a}^b\left(f(x)^{\phi'(x)}\right)^{dx}\right]^{-1}.$$
    \end{enumerate}
\end{proposition}

We will also need the following definitions and properties of multiplicative series.
\begin{definition}
Let $\{a_n \}_{n\geq 0}$ be a sequence in $\mathbb{R}_{e}$, and $x_0\in \mathbb{R}$ be a fixed point. Then the product $\prod_{n=0}^{\infty} (a_n)^{(x-x_0)^n}$ is called a multiplicative power series.
\end{definition}

\begin{theorem}\label{Jac-Mul-5}
\cite[Theorem 3.2]{Yalcin-2}
    Assume $y(x)$ is an infinitely multiplicative differentiable function at a neighbourhood of a point $x_0$. Then $y(x)$ has the multiplicative power series expansion
    \begin{equation}\label{Jac-Mul-61}
        y(x)=\prod_{n=0}^{\infty} (a_n)^{(x-x_0)^n}, \ (a_n\in \mathbb{R}_e).
    \end{equation}
\end{theorem}

Moreover, the set of real and positive functions $\{y_1(x),\ldots, y_n(x)\}$ is said to be \emph{multiplicative linearly independent} on an interval $\mathbf{I}$ if the only solution to the equation $\prod_{k=1}^ny_k^{c_k}(x)=1$ is $c_1=\cdots=c_n=0$, where $c_k$ are scalars, for every $x\in \mathbf{I}$. Since $\prod_{k=1}^ny_k^{c_k}(x)=1$ is equivalent to $\sum_{k=1}^nc_k\ln(y_k(x))=0$, we have the following result.
\begin{proposition}\label{Jac-Mul-101}
The set of real and positive functions $\{y_1(x),\ldots, y_n(x)\}$ is multiplicative linearly independent on the interval $\mathbf{I}$ if and only if the set of functions $\{\ln(y_1(x)),\ldots, \ln(y_n(x))\}$ is linearly independent on $\mathbf{I}$.
\end{proposition}

Now, we define some multiplicative algebraic structures with their algebraic operations that will be useful in the sequel. Let $\mathbf{A}$ be a non-empty subset of  $\mathbb{R}_e$. The pair $(\mathbf{A}, \oplus)$ and the triplet $(\mathbf{A}, \oplus, \odot )$ define a commutative group and ring %in multiplicative sense 
\cite{Kadak}, respectively, where $\oplus$ and $\odot$ are multiplicative algebraic operations defined by 
\begin{equation*}
  \begin{split}
     \oplus :& \ \mathbf{A}\times \mathbf{A}\longmapsto \mathbf{A} \ \   \ \ \ \ \ \ \ \ \ \ \ \ \  \ \ \ \ \ \ \odot :\mathbf{A}\times \mathbf{A} \longmapsto \mathbf{A} \\
       & \  (k,s) \ \  \longmapsto k\oplus s=ks=sk,  \ \ \ \ \ \ \ \   (k,s) \ \ 
 \longmapsto k\odot  s=k^{\ln s}=s^{\ln k}.
  \end{split}
\end{equation*}
Notice that if $f,g$ are multiplicative differentiable functions at $x$, from the fourth property of Proposition \ref{Jac-Mul-36} we get
\begin{equation*}
(f\odot g^*)(x)=\left((f^*(x))^{\ln g(x)} \right)\oplus\left(f(x)^{(\ln g(x))'} \right)= \left[(f^* \odot g )(x)\right]\oplus\left(f(x) \odot e^{(\ln g(x))'}   \right),
\end{equation*}
and using \eqref{Jac-Mul-23-b} we conclude that 
\begin{equation}\label{Jac-Mul-37}
    (f\odot g)^*(x)= \left[(f^* \odot g)(x) \right]\oplus\left[(f \odot g^*)(x)   \right].
\end{equation}
Moreover, for any function $\phi(x)$ we have
\begin{equation}\label{Jac-Mul-38}
    [(f\odot g)(x)]^{\phi(x)}=(f^{\phi(x)}\odot g)(x)=(f\odot g^{\phi(x)})(x)=[(g\odot f)(x)]^{\phi(x)}.
\end{equation}

Finally, given an interval $[a,b]$, $(a<b)$, and a positive weight function $\omega$, the $\mathbf{L}^2_{*}\left([a,b],\omega\right)$ weighted multiplicative space with respect to $\omega$ is defined by
\begin{equation}\label{Jac-Mul-201}
\mathbf{L}^2_{*}([a,b],\omega) := \left\{ f : [a,b] \to \mathbb{R}_e \ : \ \int_{a}^{b} [f(x) \odot f(x)]^{\omega(x)} \, dx < \infty \right\}.
\end{equation}
This space induces the multiplicative inner product or $^*$-inner product (see \cite{Goktas-2})
\begin{equation}\label{Jac-Mul-34-b}
  \begin{split}
     \langle \cdot ,\cdot  \rangle_{*, \omega} \ : &\ \mathbf{L}^2_{*}\times  \mathbf{L}^2_{*}\longmapsto \mathbb{R}_e \\
    &\ \  (f,g) \    \longmapsto \langle f ,g \rangle_{*,\omega}=\int_{a}^b [f(x)\odot g(x)]^{\omega(x)dx}.
  \end{split}
\end{equation}
If $f=g$ in \eqref{Jac-Mul-34-b}, we will write $\|f\|_{*,\omega}^2= \langle f,f\rangle_{*,\omega}$. The pair $(\mathbf{L}^2_{*}\left([a,b],\omega\right), \langle \cdot,\cdot \rangle_{*,\omega})$ is called a multiplicative inner product (or $*$-inner product) space. The space $\langle \cdot,\cdot \rangle_{*,\omega}$ satisfies the following properties for every $f,g\in \mathbf{L}^2_{*}\left([a,b],\omega\right)$ and $k\in \mathbb{R}$
\begin{enumerate}
  \item $\langle f,f \rangle_{*,\omega} \geq 1$, with $\langle f,f \rangle_{*,\omega}=1$ if and only if $f=1$, 
  \item $\langle f\oplus g,h \rangle_{*,\omega}=\langle f,h \rangle_{*,\omega}\oplus \langle g,h \rangle_{*,\omega}$,
  \item $\langle e^{k}\odot f,g \rangle_{*,\omega}=e^{k}\odot\langle f,g \rangle_{*,\omega}$,
  \item $\langle f,g \rangle_{*,\omega}=\langle g,f \rangle_{*,\omega}$.
\end{enumerate}

Moreover, there exists a relation between $\langle \cdot,\cdot \rangle_{*,\omega}$ and the inner product $\langle \cdot,\cdot \rangle_{\omega}$ induced by the weighted Hilbert space $\mathbf{L}^2\left([a,b],\omega\right)$, as can be seen in the following result.

\begin{lemma}\label{Jac-Mul-43}
    If $f,g\in \mathbf{L}_{*}^{2}\left([a,b],\omega\right)$, then
    \begin{equation}\label{Jac-Mul-45}
        \ln \left(  \langle f,g \rangle_{*,\omega} \right)= \langle \ln f,\ln g \rangle_{\omega}.
    \end{equation}
In particular, if $f=g$, 
\begin{equation}\label{Jac-Mul-42}
        \ln \left(  \| f\|_{*,\omega}^2 \right)=\| \ln f\|_{\omega}^2.
\end{equation}
As a consequence, $f\in \mathbf{L}_{*}^{2}\left([a,b],\omega\right)$ if and only if $\ln (f)\in \mathbf{L}^{2}\left([a,b],\omega\right)$.
\end{lemma}
\begin{proof}
The first equality follows from \eqref{Jac-Mul-29} and  \eqref{Jac-Mul-34-b}
\begin{equation*}
    \ln \left(  \langle f,g \rangle_{*,\omega} \right)=\ln \left( \exp \left( \int_a^b(\ln f) (\ln g)\omega(x)dx \right)\right)=\langle \ln f,\ln g \rangle_{\omega}.
\end{equation*}
and \eqref{Jac-Mul-42} follows by taking $f=g$.
\end{proof}

%===============================================================================================
%==========================================================================================
%==========================================================================================
%==========================================================================================

\section{The multiplicative Jacobi equation and its multiplicative polynomials}
In this section, the main goal is to introduce multiplicative Jacobi polynomials and deduce some or their properties. They will arise as one of the solutions of a particular case of the following multiplicative Sturm-Liouville equation
\begin{equation}\label{Jac-Mul-2}
    \frac{d^*}{dx}\left( e^{p(x)}\odot \frac{d^*y}{dx} \right)\oplus
    \left( e^{q(x)}\odot y \right)\oplus \left(e^{ \gamma \omega(x)}\odot y\right)=1, \ \ x\in[-1,1],
\end{equation}
where $\gamma$ is a spectral parameter, and $p(x)$, $q(x)$, $\omega(x)$ are real-valued continuous functions (for more properties and examples of the multiplicative Sturm-Liouville equation see \cite{Goktas-2}). 

Taking $\gamma=r(r+\alpha+\beta+1)$, $\omega(x)=(1-x)^{\alpha}(1+x)^{\beta}$ with $\alpha , \beta>-1$, $r\in \mathbb{R}$, $q(x)=0$ and $p(x)=(1-x^2)\omega(x)$, the multiplicative equation \eqref{Jac-Mul-2} becomes
\begin{equation}\label{Jac-Mul-25}
    \frac{d^*}{dx}\left( \left( y^*\right)^{(1-x^2)(1-x)^{\alpha}(1+x)^{\beta}} \right)y^{r(r+\alpha+\beta+1) (1-x)^{\alpha}(1+x)^{\beta}}=1,  \ \ x\in[-1,1],
\end{equation}
or, equivalently,
\begin{equation}\label{Jac-Mul-1}
    \left(y^{**} \right)^{1-x^2}\left(y^{*} \right)^{\beta-\alpha-(\alpha+\beta+2)x}y^{r(r+\alpha+\beta+1)}=1, \ \ x\in[-1,1].
\end{equation}
In particular, if we set $r=n$ as a non-negative integer, \eqref{Jac-Mul-1} is called an $n$-th order \emph{multiplicative Jacobi equation}. In such a case, $y(x,r=n)>0$ is a solution of \eqref{Jac-Mul-1}, and it is called a multiplicative Jacobi solution. Some special cases are:
\begin{enumerate}
      \item The \emph{multiplicative Legendre equation}, when $\alpha=\beta=0$ \cite{Goktas}.
  \item The \emph{multiplicative Chebyshev equation} of first kind, when $\alpha=\beta=-1/2$ \cite{Yalcin}.
  \item The \emph{multiplicative Chebyshev equation} of second kind, when $\alpha=\beta=1/2$ \cite{Yalcin}.
  \item The \emph{multiplicative Gegenbauer (or Ultraspherical) equation}, when $\alpha=\beta \neq-1/2$.
\end{enumerate}

Now, we present the solution of the multiplicative equation \eqref{Jac-Mul-1} as multiplicative power series \eqref{Jac-Mul-61} with $x_0=0$.  
\begin{proposition}
    Let $a_0$ and $a_1$ be arbitrary positive numbers. The general solution in multiplicative power series for the multiplicative differential equation \eqref{Jac-Mul-1} is 
    \begin{equation}\label{Jac-Mul-9}
        y(x)= a_0\prod_{k=0}^{\infty}(a_0)^{m_{k}(\lambda)x^{k+2}}(a_1)^x\prod_{k=1}^{\infty}(a_1)^{q_{k}(\lambda)x^{k+1}},
    \end{equation}
    with $\{m_{n}\}_{n\geq 0}$ and $\{q_{n}\}_{n\geq 1}$  given by 
\begin{equation}\label{Mul-Jac-4}
\begin{split}
    m_{n-1}(\lambda)&:=\frac{\mu_0}{(n+1)!}\left[\lambda^{n-1}+\lambda^{n-3}\sum_{k=2}^{n-1}\mu_k+\lambda^{n-5}\sum_{k=2}^{n-3}\left(\sum_{j=2}^{k}\mu_j\right)\mu_{k+2}+\cdots \right], \ \ n\geq 1,\\
    q_n(\lambda)&:=\frac{1}{(n+1)!}\left[\lambda^{n}+\lambda^{n-2}\sum_{k=1}^{n-1}\mu_k+\lambda^{n-4}\sum_{k=1}^{n-3}\left(\sum_{j=1}^{k}\mu_j\right)\mu_{k+2}+\cdots \right], \ \ n\geq 1,
\end{split}
\end{equation}
and
\begin{equation}\label{Mul-Jac-9-b}
\begin{split}
    m_{n-1}(0)&= \left\{ \begin{array}{lcc}  \frac{1}{(n+1)!}\prod_{k=0}^{(n-1)/2}\mu_{2k}, & if & n \mbox{ is odd}, \\ \\ 0, & if & n \mbox{ is even}, \\  \end{array} \right.\\
    q_n(0)&= \left\{ \begin{array}{lcc}  \frac{1}{(n+1)!}\prod_{k=1}^{n/2}\mu_{2k-1}, & if & n \mbox{ is even}, \\ \\ 0, & if & n \mbox{ is odd}, \\  \end{array} \right.
\end{split}
\end{equation}
where $\lambda:=\alpha-\beta$ and $\mu_n \equiv \mu_n(n,\alpha,\beta,r):=(n-r)(r+\alpha+\beta+n+1)$, for all $n\geq 0$. 

%Moreover, the sequence $\{p_n(\lambda)=q_{n}(\lambda)+m_{n-1}(\lambda)\}_{n\geq 1}$ satisfies the following recurrence relation
%    \begin{equation*}
%    \begin{split}
%        p_1(\lambda)&=\frac{1}{2!}(\lambda+\mu_0),\\
%        p_2(\lambda)&=\frac{1}{3!}(\lambda^2+\mu_0\lambda+\mu_1),\\
%        p_{n+1}(\lambda)&=\frac{1}{n+2}\lambda p_n(\lambda)+\frac{\mu_n}{(n+2)(n+1)}p_{n-1}(\lambda), \ \ n\geq 2.
%    \end{split}
%    \end{equation*}
\end{proposition}
\begin{proof}
Assume that the solution of \eqref{Jac-Mul-1} is infinitely multiplicative differentiable at a neighbourhood of $x_0=0$, so Theorem \ref{Jac-Mul-5} guarantees that \eqref{Jac-Mul-1} has a solution in multiplicative power series in the form \eqref{Jac-Mul-61}, i.e.,
\begin{equation}\label{Jac-Mul-6}
    y(x)=\prod_{n=0}^{\infty} (a_n)^{x^n}, \ \ (a_n\in \mathbb{R}_e).
\end{equation}
In fact, since $x_0=0$ is a multiplicative ordinary point from \eqref{Jac-Mul-1}, then from \cite[Theorem 3.1]{Yalcin-2}, there exists a multiplicative power series solution as in \eqref{Jac-Mul-6}. Taking the multiplicative derivatives in \eqref{Jac-Mul-6} and substituting in \eqref{Jac-Mul-1} we get
\begin{equation*}
\begin{split}
    1=&\left[ \prod_{n=2}^{\infty} (a_n)^{n(n-1)x^{n-2}} \right]^{1-x^2}
    \left[ \prod_{n=1}^{\infty} (a_n)^{nx^{n-1}} \right]^{\beta-\alpha-(\alpha+\beta+2)x} 
    \left[ \prod_{n=0}^{\infty} (a_n)^{x^{n}} \right]^{r(r+\alpha+\beta+1)}\\
    =&(a_1)^{\beta-\alpha-(\alpha+\beta+2)x)}
    (a_0)^{r(r+\alpha+\beta+1)}(a_1)^{r(r+\alpha+\beta+1)x}\\
    &\prod_{n=2}^{\infty} (a_n)^{n(n-1)x^{n-2}+n(\beta-\alpha)x^{n-1}
+[n(1-n)-n(\alpha+\beta+2)+r(r+\alpha+\beta+1)]x^n}\\
=&\left[(a_2)^2(a_1)^{\beta-\alpha}(a_0)^{r(r+\alpha+\beta+1)}\right]
\left[(a_3)^6(a_2)^{2(\beta-\alpha)}(a_1)^{r(r+\alpha+\beta+1)-(\alpha+\beta+2)}\right]^x
\\
&\prod_{n=2}^{\infty}\left[ (a_{n+2})^{(n+2)(n+1)}(a_{n+1})^{(n+1)(\beta-\alpha)}(a_n)^{
-n(n+\alpha+\beta+1)+r(r+\alpha+\beta+1)} \right]^{x^n}.
\end{split}
\end{equation*}
Since $r(r+\alpha+\beta+1)-(\alpha+\beta+2)=(r-1)(r+\alpha+\beta+2)$ and $-n(n+\alpha+\beta+1)+r(r+\alpha+\beta+1)=-(n-r)(r+\alpha+\beta+n+1)$, we get
\begin{equation*}
    \begin{split}
        1=& (a_2)^2(a_1)^{\beta-\alpha}(a_0)^{r(r+\alpha+\beta+1)},\\
        1=&(a_3)^6(a_2)^{2(\beta-\alpha)}(a_1)^{(r-1)(r+\alpha+\beta+2)},\\
        1=&(a_{n+2})^{(n+2)(n+1)}(a_{n+1})^{(n+1)(\beta-\alpha)}(a_n)^{-(n-r)(r+\alpha+\beta+n+1)},  \ \ n\geq 2.
    \end{split}
\end{equation*}
For $a_0$, $a_1$ arbitrary positive numbers and solving and substituting with $\lambda=\alpha-\beta$ and $\mu_n=(n-r)(r+\alpha+\beta+n+1)$, we obtain the following recurrence relation
\begin{equation}\label{Mul-Jac-2}
    \begin{split}
a_2&=(a_1)^{\frac{1}{2!}\lambda}(a_0)^{\frac{1}{2!}\mu_0},\\
a_3&=(a_1)^{\frac{1}{3!}[\lambda^2+\mu_1]}(a_0)^{\frac{1}{3!}\mu_0\lambda},\\
a_{n+2}&=(a_{n+1})^{\frac{1}{n+2}\lambda}(a_n)^{\frac{1}{(n+2)(n+1)}\mu_n}, \ \ n\geq 2.
    \end{split}
\end{equation}
Using this recurrence relation, we can write all the elements of the sequence in terms of $a_0$ and $a_1$. Indeed, for the even powers we have
\begin{equation*}
    \begin{split}
a_2&=(a_1)^{\frac{\lambda}{2!}}(a_0)^{\frac{\mu_0}{2!}},\\
a_{4}&=(a_1)^{\frac{\lambda}{4!}[\lambda^2+(\mu_1+\mu_2)]}(a_0)^{\frac{\mu_0}{4!}[\lambda^2+\mu_2]},\\
a_{6}&=(a_1)^{\frac{\lambda}{6!}\left[\lambda^4+\lambda^2(\mu_1+\mu_2+\mu_3+\mu_4)+(\mu_1\mu_3+(\mu_1+\mu_2)\mu_4)\right]}(a_0)^{\frac{\mu_0}{6!}\left[\lambda^4+\lambda^2(\mu_2+\mu_3+\mu_4)+\mu_2\mu_4\right]},\\
%a_{8}&=(a_1)^{\frac{\lambda}{8!}[\lambda^6+\lambda^4(\mu_1+\mu_2+\mu_3+\mu_4+\mu_5+\mu_6)+\lambda^2(\mu_1\mu_3+(\mu_1+\mu_2)\mu_4+(\mu_1+\mu_2+\mu_3)\mu_5+(\mu_1+\mu_2+\mu_3+\mu_4)\mu_6)+\mu_1\mu_3\mu_5]}\\
%&\ \ \ \ (a_0)^{\frac{\lambda\mu_0}{7!}[\lambda^4+\lambda^2(\mu_2+\mu_3+\mu_4+\mu_5)+(\mu_2\mu_4+(\mu_2+\mu_3)\mu_5)]},\\
 & \ \  \vdots\\
a_{2n}&=(a_1)^{\frac{\lambda}{(2n)!}\left[\lambda^{2n-2}+\lambda^{2n-4}\sum_{k=1}^{2n-2}\mu_k+\lambda^{2n-6}\sum_{k=1}^{2n-4}\left(\sum_{j=1}^{k}\mu_j\right)\mu_{k+2}+\cdots \right]}\\
&\ \ \ \ (a_0)^{\frac{\mu_0}{(2n)!}\left[\lambda^{2n-2}+\lambda^{2n-4}\sum_{k=2}^{2n-2}\mu_k+\lambda^{2n-6}\sum_{k=2}^{2n-4}\left(\sum_{j=2}^{k}\mu_j\right)\mu_{k+2}+\cdots +\prod_{k=1}^{n-1}\mu_{2k}\right]}.
%\\
% & \ \  \vdots\\
    \end{split}
\end{equation*}
In similar way, for the odd powers we get
\begin{equation*}
    \begin{split}
a_3&=(a_1)^{\frac{1}{3!}[\lambda^2+\mu_1]}(a_0)^{\frac{ \mu_0}{3!}\lambda},\\
a_{5}&=(a_1)^{\frac{1}{5!}[\lambda^4+\lambda^2(\mu_1+\mu_2+\mu_3)+\mu_1\mu_3]}(a_0)^{\frac{\lambda\mu_0}{5!}[\lambda^2+(\mu_2+\mu_3)]},\\
a_{7}&=(a_1)^{\frac{1}{7!}[\lambda^6+\lambda^4(\mu_1+\mu_2+\mu_3+\mu_4+\mu_5)+\lambda^2(\mu_1\mu_3+(\mu_1+\mu_2)\mu_4+(\mu_1+\mu_2+\mu_3)\mu_5)+\mu_1\mu_3\mu_5]}\\
&\ \ \ \ (a_0)^{\frac{\lambda\mu_0}{7!}[\lambda^4+\lambda^2(\mu_2+\mu_3+\mu_4+\mu_5)+(\mu_2\mu_4+(\mu_2+\mu_3)\mu_5)]},\\
 & \ \  \vdots\\
a_{2n+1}&=(a_1)^{\frac{1}{(2n+1)!}\left[\lambda^{2n}+\lambda^{2n-2}\sum_{k=1}^{2n-1}\mu_k+\lambda^{2n-4}\sum_{k=1}^{2n-3}\left(\sum_{j=1}^{k}\mu_j\right)\mu_{k+2}+\cdots +\prod_{k=1}^{n}\mu_{2k-1}\right]}\\
&\ \ \ \ (a_0)^{\frac{\lambda\mu_0}{(2n+1)!}\left[\lambda^{2n-2}+\lambda^{2n-4}\sum_{k=2}^{2n-1}\mu_k+\lambda^{2n-6}\sum_{k=2}^{2n-3}\left(\sum_{j=2}^{k}\mu_j\right)\mu_{k+2}+\cdots \right]}.
%\\
% & \ \  \vdots
    \end{split}
\end{equation*}
In general, the power of $a_{n+1}$ is a polynomial of degree $n$ in the variable $\lambda$. Taking into account \eqref{Mul-Jac-4}, if we denote such polynomial by $p_n(\lambda)$, then we have
\begin{equation}\label{Jac-Mul-20}
    p_n(\lambda):=q_n(\lambda)+m_{n-1}(\lambda), \ \ n\geq 1.
\end{equation}
From the powers of $a_{2n}$ and $a_{2n+1}$, it follows that
\begin{equation*}%\label{Mul-Jac-7}
    p_n(0)= \left\{ \begin{array}{lcc}   q_n(0)=\frac{1}{(n+1)!}\prod_{k=1}^{n/2}\mu_{2k-1}, & if & n \mbox{ is even}, \\ \\ m_{n-1}(0)=\frac{1}{(n+1)!}\prod_{k=0}^{(n-1)/2}\mu_{2k}, & if & n \mbox{ is odd}. \\  \end{array} \right.
\end{equation*}
Besides, $m_{n-1}(0)=0$ if $n$ is even and $q_{n}(0)=0$ if $n$ is odd, so we get \eqref{Mul-Jac-9-b}. With this notation, we have 
\begin{equation}\label{Jac-Mul-21}
a_{n+1}=a_{1}^{q_{n}(\lambda)}a_0^{m_{n-1}(\lambda)}, \ \     n\geq 1.
\end{equation}
Substituting $a_{n+1}$ in \eqref{Jac-Mul-6} and rearranging we get \eqref{Jac-Mul-9}.
\end{proof}
The sequence of polynomials $\{p_n(\lambda)\}_{n\geq 1}$ in \eqref{Jac-Mul-20} has the following interesting property. 
\begin{proposition}\label{Mul-Jac-6}
    The sequence $\{p_n(\lambda) \}_{n\geq 1}$ satisfies the following recurrence relation
    \begin{equation*}
    \begin{split}
        p_1(\lambda)&=\frac{1}{2!}(\lambda+\mu_0),\\
        p_2(\lambda)&=\frac{1}{3!}(\lambda^2+\mu_0\lambda+\mu_1),\\
        p_{n+1}(\lambda)&=\frac{1}{n+2}\lambda p_n(\lambda)+\frac{\mu_n}{(n+2)(n+1)}p_{n-1}(\lambda), \ \ n\geq 2.
    \end{split}
    \end{equation*}
\end{proposition}
\begin{proof}
%From \eqref{Jac-Mul-21} we have that $a_{n+1}=a_{1}^{q_{n}(\lambda)}a_0^{m_{n-1}(\lambda)}$ for $n\geq 1$. 
Notice that from \eqref{Jac-Mul-21} we have
\begin{equation}\label{Mul-Jac-3}
    a_{n+2}=a_1^{q_{n+1}(\lambda)}a_0^{m_{n}(\lambda)}, \ \ n\geq 0.
\end{equation}
On the other hand, using the two first equations from \eqref{Mul-Jac-2} and comparing with \eqref{Mul-Jac-3} we get the expressions for $p_1(\lambda)$ and $p_2(\lambda)$. In a similar way, using the third equation from \eqref{Mul-Jac-2} we obtain
\begin{equation*}
    \begin{split}
a_{n+2}&=\left((a_{1})^{q_n(\lambda)}(a_{0})^{m_{n-1}(\lambda)}\right)^{\frac{1}{n+2}\lambda}\left( (a_{1})^{q_{n-1}(\lambda)}(a_{0})^{m_{n-2}(\lambda)}\right)^{\frac{\mu_n}{(n+2)(n+1)}}\\
&= (a_1)^{\frac{1}{n+2}\lambda q_n(\lambda)+\frac{\mu_n}{(n+2)(n+1)} q_{n-1}(\lambda)}(a_0)^{\frac{1}{n+2}\lambda m_{n-1}(\lambda)+\frac{\mu_n}{(n+2)(n+1)} m_{n-2}(\lambda)}
, \ \ n\geq 2,
    \end{split}
\end{equation*}
and comparing with \eqref{Mul-Jac-3}, we get 
\begin{equation*}
\begin{split}
    q_{n+1}(\lambda)&=\frac{1}{n+2}\lambda q_n(\lambda)+\frac{\mu_n}{(n+2)(n+1)} q_{n-1}(\lambda), \\ \  m_n(\lambda)&=\frac{1}{n+2}\lambda m_{n-1}(\lambda)+\frac{\mu_n}{(n+2)(n+1)} m_{n-2}(\lambda).
\end{split}
\end{equation*}
  Since $p_{n+1}(\lambda)=q_{n+1}(\lambda)+m_{n}(\lambda)$, the result is immediate. 
\end{proof}
Notice that the previous result allows the efficient computation of $q_{n}(\lambda)$ and $m_{n-1}(\lambda)$. Moreover, the sequence $\{p_n\}_{n\geq 1}$, with the initial condition $p_0(\lambda)=1$, is actually an orthogonal sequence, as the next corollary shows. 
\begin{corollary}
If $p_0(\lambda)=1$ and $\mu_n\neq 0$ (equiv. $r\neq n$ or $r\neq-(\alpha+\beta+n+1)$) for every non-negative integer $n$, then $\{p_n\}_{n\geq 0}$ is a sequence of orthogonal polynomials with respect to a regular functional. Furthermore, if $\mu_n > 0$ for every non-negative integer $n$, then the regular functional is positive definite. Moreover, the following three-term recurrence relation holds
\begin{equation*}%\label{Mul-Jac-9}
    \begin{split}
    p_0(\lambda)&=1,\\
    p_1(\lambda)&=\frac{1}{2!}(\lambda+\mu_0),\\
        p_{n+1}(\lambda)&=\frac{1}{n+2}\lambda p_n(\lambda)+\frac{\mu_n}{(n+2)(n+1)}p_{n-1}(\lambda), \ \ n\geq 1.
    \end{split}
    \end{equation*}
\end{corollary}
\begin{proof}
    The result follows at once from Proposition \ref{Mul-Jac-6} with $p_0(\lambda)=1$ and the Favard's theorem (see \cite[Theorem 4.4]{Chihara}).
\end{proof}

 In particular, if $\lambda=\alpha-\beta=0$, taking into account \eqref{Jac-Mul-9} and \eqref{Mul-Jac-9-b} we have the following result.
 
\begin{corollary}
    Let $a_0$ and $a_1$ be arbitrary positive numbers. The general solution in multiplicative power series of the multiplicative equation \eqref{Jac-Mul-1} with $\lambda=\alpha-\beta=0$ is
\begin{equation}\label{Jac-Mul-22}
        y(x)=a_0\prod_{k=1}^{\infty}(a_0)^{\prod_{j=0}^{k-1} (2j-r)(2j+r+2\alpha+1)\frac{x^{2k}}{(2k)!}}(a_1)^x\prod_{k=1}^{\infty}(a_1)^{\prod_{j=0}^{k-1} (2j+1-r)(2j+r+2\alpha+2)\frac{x^{2k+1}}{(2k+1)!}}.
    \end{equation}
\end{corollary}
From the previous corollary we have the following particular cases when $r=n$ is a non-negative integer:
\begin{enumerate}
  \item For $\alpha=\beta=0$, 
  \begin{equation*}
        y(x)=a_0\prod_{k=1}^{\infty}(a_0)^{\prod_{j=0}^{k-1} \left[2j(2j+1)-r(r+1)\right]\frac{x^{2k}}{(2k)!}}(a_1)^x\prod_{k=1}^{\infty}(a_1)^{\prod_{j=0}^{k-1} \left[2(j+1)(2j+1)-r(r+1)\right]\frac{x^{2k+1}}{(2k+1)!}},
    \end{equation*}
we get the solution from the multiplicative Legendre equation. 
  \item For $\alpha=\beta=-\frac{1}{2}$, 
  \begin{equation*}
        y(x)=a_0\prod_{k=1}^{\infty}(a_0)^{\prod_{j=0}^{k-1} \left[(2j)^2-r^2\right]\frac{x^{2k}}{(2k)!}}(a_1)^x\prod_{k=1}^{\infty}(a_1)^{\prod_{j=0}^{k-1} \left[(2j+1)^2-r^2\right]\frac{x^{2k+1}}{(2k+1)!}},
    \end{equation*}
we obtain the solution from the multiplicative Chebyshev equation of the first kind. 
  \item For $\alpha=\beta=\frac{1}{2}$,
\begin{equation*}
        y(x)=a_0\prod_{k=1}^{\infty}(a_0)^{\prod_{j=0}^{k-1} \left[4j(j+1)-r(r+2)\right]\frac{x^{2k}}{(2k)!}}(a_1)^x\prod_{k=1}^{\infty}(a_1)^{\prod_{j=0}^{k-1} \left[(2j+1)(2j+3)-r(r+2)\right]\frac{x^{2k+1}}{(2k+1)!}},
    \end{equation*}
we have the solution from the multiplicative Chebyshev equation of the second kind. 
\end{enumerate}
Notice that these particular cases coincide with those studied in \cite{Goktas,Yalcin}. As a consequence, \eqref{Jac-Mul-9} and \eqref{Jac-Mul-22} constitute a generalization of the previous contributions. On the other hand, notice that if $r=n$ is a non-negative integer, one of the infinite products appearing in \eqref{Jac-Mul-22}, which we will denote by
\begin{equation*}
\begin{split}
    F_n(x)&=a_0\prod_{k=1}^{\infty}(a_0)^{\prod_{j=0}^{k-1} (2j-n)(2j+n+2\alpha+1)\frac{x^{2k}}{(2k)!}}, \\ G_n(x)&=(a_1)^x\prod_{k=1}^{\infty}(a_1)^{\prod_{j=0}^{k-1} (2j+1-n)(2j+n+2\alpha+2)\frac{x^{2k+1}}{(2k+1)!}},
\end{split}
\end{equation*}
has a finite number of terms. Indeed, if $n$ is an even integer then $F_n(x)$ has a finite number of terms, and if $n$ is an odd integer then $G_n(x)$ has a finite number of terms. As a consequence, the power appearing on either $F_n(x)$ or $G_n(x)$ is a polynomials in the variable $x$, i.e., one of these products is a multiplicative polynomial. This result is valid not only when $\alpha=\beta$, but also when $\alpha\neq\beta$, as the next result shows.

\begin{theorem}\label{Jac-Mul-100}
%Let $\alpha > -1$, $\beta > -1$. 
Consider the multiplicative equation 
    \begin{equation}\label{Jac-Mul-8}
    \left(y^{**} \right)^{1-x^2}\left(y^{*} \right)^{\beta-\alpha-(\alpha+\beta+2)x}y^{\gamma}=1, \ \ x\in[-1,1],
\end{equation}
where $\gamma \in \mathbb{R}$ is a spectral parameter. The equation \eqref{Jac-Mul-8} has a multiplicative polynomial solution, not identically one, of the form $y(x)=\prod_{k=0}^{n}(a_k)^{(x-1)^k}$ ($a_k\in\mathbb{R}_e$) if and only if $\gamma=n(n+\alpha+\beta+1)$ for every non-negative integer $n$. Moreover, the solution $y^c$ ($c$ constant) is the only multiplicative polynomial solution, i.e. every other solution of  \eqref{Jac-Mul-8} that is linearly independent to $y$ on $(-1,1)$ is not a multiplicative polynomial.
\end{theorem}
\begin{proof}
Assume that the solution of \eqref{Jac-Mul-8} is infinitely multiplicative differentiable at a neighbourhood of $x_0=1$, so Theorem \ref{Jac-Mul-5} guarantees that\eqref{Jac-Mul-1} has a solution in multiplicative power series of the form $y(x)=\prod_{k=0}^{\infty} (a_k)^{(x-1)^k}.$ Substituting in \eqref{Jac-Mul-8}, we get
\begin{equation*}
  \begin{split}
     1&=  \prod_{k=2}^{\infty}(a_k)^{k(k-1)(1-x^2)(x-1)^{k-2}}
     \prod_{k=1}^{\infty}(a_k)^{k(\beta-\alpha-(\alpha+\beta+2)x)(x-1)^{k-1}}\prod_{k=0}^{\infty}(a_k)^{\gamma(x-1)^k}\\
&=\left(a_0^{\gamma}a_1^{-2(\alpha+1)} \right)^{(x-1)^0}
\left(a_1^{\gamma-(\alpha+\beta+2)}a_2^{-4(\alpha+2)} \right)^{(x-1)}\\
&\times
\prod_{k=2}^{\infty}\left((a_k)^{\gamma-k(k+\alpha+\beta+1)} (a_{k+1})^{-2(k+1)(k+\alpha+1)}\right)^{(x-1)^k}  \\
       & =\prod_{k=0}^{\infty}\left((a_k)^{\gamma-k(k+\alpha+\beta+1)} (a_{k+1})^{-2(k+1)(k+\alpha+1)}\right)^{(x-1)^k}.
  \end{split}
\end{equation*}
As a consequence, we have 
\begin{equation}\label{Mul-Jac-10}
    (a_k)^{\gamma-k(k+\alpha+\beta+1)} (a_{k+1})^{-2(k+1)(k+\alpha+1)}=1, \ \ k=0,1,2,\ldots.
\end{equation}

Assuming that $y(x)=\prod_{k=0}^{n}(a_k)^{(x-1)^k}$, then $a_n$ is the last non-one coefficient, i.e., $a_{k+1}=1$ for $k\geq n$. Taking $k=n$ in \eqref{Mul-Jac-10} we deduce that the exponent of $a_n$ must be $\gamma=n(n+\alpha+\beta+1)$. Conversely, if 
$\gamma=n(n+\alpha+\beta+1)$, then from \eqref{Mul-Jac-10} we get $a_{k+1}=1$ for $k\geq n$.

Moreover, assume that $y(x)=\prod_{k=0}^{n}(a_k)^{(x-1)^k}$ and $z(x)$ are both multiplicative solutions of \eqref{Jac-Mul-8} or, equivalently, are solutions of \eqref{Jac-Mul-25} with $r=n$. In such a case, 
\begin{equation*}
    \left[ (y^*)^{(1-x^2)\omega(x)}\right]^*\oplus y^{n(n+\alpha+\beta+1)\omega(x)}=1, \ \ \left[ (z^*)^{(1-x^2)\omega(x)}\right]^*\oplus z^{n(n+\alpha+\beta+1)\omega(x)}=1,
\end{equation*}
since $1\odot z=1\odot y=1$, using the distributive property of $\odot$, we obtain
\begin{equation*}
    1=\frac{\left(\left[ (y^*)^{(1-x^2)\omega(x)}\right]^*\odot z\right)
    \oplus \left(y^{n(n+\alpha+\beta+1)\omega(x)} \odot z\right)}{\left(\left[ (z^*)^{(1-x^2)\omega(x)}\right]^*\odot y\right)
    \oplus \left(z^{n(n+\alpha+\beta+1)\omega(x)} \odot y\right)}.
\end{equation*}
From \eqref{Jac-Mul-38}, $y^{n(n+\alpha+\beta+1)\omega(x)} \odot z=z^{n(n+\alpha+\beta+1)\omega(x)} \odot y$, and then 
\begin{equation}\label{Jac-Mul-41}
    1=\frac{\left(\left[ (y^*)^{(1-x^2)\omega(x)}\right]^*\odot z\right)
    }{\left(\left[ (z^*)^{(1-x^2)\omega(x)}\right]^*\odot y\right)}.
\end{equation}
On the other hand, from \eqref{Jac-Mul-37} we get 
\begin{equation*}
     \frac{\left((y^*)^{(1-x^2)\omega(x)}\odot z\right)^*}{\left((z^*)^{(1-x^2)\omega(x)}\odot y\right)^*} =\frac{\left(\left[ (y^*)^{(1-x^2)\omega(x)}\right]^*\odot z\right)
    \oplus \left((y^*)^{(1-x^2)\omega(x)} \odot z^*\right)}{\left(\left[ (z^*)^{(1-x^2)\omega(x)}\right]^*\odot y\right)
    \oplus \left((z^*)^{(1-x^2)\omega(x)} \odot y^*\right)}.
\end{equation*}
Using \eqref{Jac-Mul-38}, $(y^*)^{(1-x^2)\omega(x)} \odot z^*=(z^*)^{(1-x^2)\omega(x)} \odot y^*$, and from  \eqref{Jac-Mul-41} we conclude
\begin{equation*}
    \left( \frac{(y^*\odot z)^{(1-x^2)\omega(x)}}{(z^*\odot y)^{(1-x^2)\omega(x)}} \right)^*=\left( \frac{(y^*)^{(1-x^2)\omega(x)}\odot z}{(z^*)^{(1-x^2)\omega(x)}\odot y} \right)^*=1.
\end{equation*}
From Proposition \ref{Jac-Mult-24}, there exists $C\in \mathbb{R}_e$ such that $ \left(   \frac{y^*\odot z}{z^*\odot y} \right)^{(1-x^2)\omega(x)}=C$, which is equivalent to
\begin{equation*}%\label{Jac-Mul-102}
  (1-x^2)\omega(x)\left[  (\ln y)' \ln z-(\ln z)' \ln y\right]=(1-x^2)\omega(x)\left|
  \begin{array}{cc}
    \ln z & \ln y \\
     (\ln z)' & (\ln y)' \\
  \end{array}
\right|
  =\ln C.
\end{equation*} 
If $\ln C =0$, it follows that $\ln y$ and $\ln z$ are linearly dependent on $(-1,1)$, i.e., $\ln z$ is a constant multiple of $\ln y= \sum_{k=0}^n \ln(a_k)(x-1)^k$, i.e. $\ln z$ is a polynomial. If $\ln C \neq0$, we conclude that $\ln y$ and $\ln z$ are linearly independent on $(-1,1)$. Following \cite[Theorem $4.2.2$]{Szego}, taking $ x\rightarrow \pm 1$ we find that $\ln y$ and $\ln z$ cannot be both polynomials. Using Proposition \ref{Jac-Mul-101}, it follows that if $y$ and $z$ are linearly independent on $(-1,1)$, $z$ cannot be a multiplicative polynomial.
\end{proof}

As a consequence, the multiplicative Jacobi equation \eqref{Jac-Mul-1} (with $r=n$) has a multiplicative polynomial solution of the form 
\begin{equation}\label{Jac-Mul-24}
    y_n(x)=\prod_{k=0}^n(a_k)^{(x-1)^k}=a_0\prod_{k=1}^n(a_k)^{(x-1)^k},
\end{equation}
for every non-negative integer $n$. From \eqref{Mul-Jac-10} and taking into account that $n(\alpha+\beta+n+1)-(k-1)(\alpha+\beta+k)=(n-(k-1))(\alpha+\beta+n+k)$, we get for $k=1,2\ldots,n$
\begin{equation*}
\begin{split}
    a_{k}&=(a_{k-1})^{\frac{(n-(k-1))(\alpha+\beta+n+k)}{2(k)(\alpha+k)}}\\
    &=(a_{k-2})^{\frac{[(n-(k-1))(n-(k-2))][(\alpha+\beta+n+k)(\alpha+\beta+n+k-1)]}{2^2[(k)(k-1)][(\alpha+k)(\alpha+k-1)]}}\\
    & \ \ \vdots \\
    &=(a_{0})^{\frac{[(n-(k-1))(n-(k-2))\cdots (n)][(n+\alpha+\beta+k)(n+\alpha+\beta+k-1)\cdots(n+\alpha+\beta+1)]}{2^k[k(k-1)\cdots 1][(\alpha+k)(\alpha+k-1)\cdots (\alpha+1)]}},\\
    &=(a_{0})^{\frac{1}{2^k}\binom{n}{k}\frac{(\alpha+\beta+n+1)_k}{(\alpha+1)_k} }.
\end{split}
\end{equation*}
Substituting in \eqref{Jac-Mul-24}, for $n\geq1$ we obtain 
\begin{equation*}
    y_n(x)=(a_0)^{1+\sum_{k=1}^{n-1}\binom{n}{k}\frac{(\alpha+\beta+n+1)_{k}}{(\alpha+1)_k}\left( \frac{x-1}{2} \right)^k+\frac{(\alpha+\beta+n+1)_{n}}{(\alpha+1)_n}\left( \frac{x-1}{2} \right)^n}=(a_0)^{\frac{n!}{(\alpha+1)_n}P_n^{(\alpha,\beta)}(x)},
\end{equation*}
where $P_n^{(\alpha,\beta)}(x)$ is the Jacobi classical polynomial, as in \eqref{Jac-Mul-26}, and $(\cdot)_n$ is the Pochhammer symbol as in \eqref{Jac-Mul-27}. Since $y_n$ being a solution of  \eqref{Jac-Mul-1} implies that  $y^c$ ($c\in \mathbb{R}$) is also a solution, it is customary to choose an appropriate value for $a_0$. In particular, if we take $a_0=e^{\frac{(\alpha+1)_n}{n!}}$, then the solution $y_n(x)$ will be called the \emph{multiplicative Jacobi polynomial} of degree $n$ and will be denoted by 
%\begin{equation}\label{Jac-Mul-28}
%    \overset{*}{P}_n^{(\alpha,\beta)}(x)=e^{P_n^{(\alpha,\beta)}(x)}, \ \ n\geq 0.
%\end{equation}
\begin{equation}\label{Jac-Mul-28}
    \tilde{P}_n^{(\alpha,\beta)}(x)=\exp\left(P_n^{(\alpha,\beta)}(x)\right)=e^{P_n^{(\alpha,\beta)}(x)}, \ \ n\geq 0,
\end{equation}
or, equivalently, 
\begin{equation}\label{Jac-Mul-28-b}
    P_n^{(\alpha,\beta)}(x)=\ln \left(\tilde{P}_n^{(\alpha,\beta)}(x)\right), \ \ n\geq 0.
\end{equation}

%===============================================================================================
%==========================================================================================
%==========================================================================================
%==========================================================================================

\section{Some properties of multiplicative Jacobi polynomials}

In this section, we deduce some properties of multiplicative Jacobi polynomials that are similar to those of the classical Jacobi polynomials. In particular, we also prove some relations for multiplicative Legendre polynomials and the multiplicative Chebyshev polynomials of the first and second kind.

\begin{proposition}\label{Jac-Mul-39}
The sequence $\{ \tilde{P}_n^{(\alpha,\beta)}\}_{n\geq 0}$ of multiplicative Jacobi polynomials are multiplicative orthogonal on $[-1,1]$ with respect to $\omega(x)=(1-x)^{\alpha}(1+x)^{\beta}$ ($\alpha > -1$, $\beta > -1$). Furthermore, 
\begin{equation*}
    \int_{-1}^1 \left(\tilde{P}_n^{(\alpha,\beta)}(x)\odot \tilde{P}_m^{(\alpha,\beta)}(x)\right)^{\omega(x)dx}=\exp\left(\frac{2^{\alpha+\beta+1}}{2n+\alpha+\beta+1}\frac{\Gamma(n+\alpha+1)\Gamma(n+\beta+1)}{n!\Gamma(n+\alpha+\beta+1)}\delta_{n,m}\right).
\end{equation*}
\end{proposition}
\begin{proof}
Suppose $n$ and $m$ are non-negative integers. From \eqref{Jac-Mul-29} and \eqref{Jac-Mul-28-b} 
    \begin{equation*}
        \begin{split}
            \int_{-1}^1 \left(\tilde{P}_n^{(\alpha,\beta)}(x)\odot \tilde{P}_m^{(\alpha,\beta)}(x)\right)^{\omega(x)dx}&= e^{\int_{-1}^1 \ln \left[ \left(\tilde{P}_n^{(\alpha,\beta)}(x)\right)^{P_m^{(\alpha,\beta)}(x)\omega(x)}  \right]dx}\\
            &=e^{\int_{-1}^1P_n^{(\alpha,\beta)}(x)P_m^{(\alpha,\beta)}(x)\omega(x)dx},
        \end{split}
    \end{equation*}
the required result follows from the orthogonality of classical Jacobi polynomials \eqref{Jac-Mul-30}.
\end{proof}
The interval of orthogonality can be generalized to the interval $[a,b]$, $(a<b)$, by taking $\tilde{P}_n^{(\alpha,\beta)}\left(2\frac{x-a}{b-a}-1\right)$ and $\omega(x)=(b-x)^{\alpha}(a+x)^{\beta}$. The multiplicative Jacobi polynomials satisfy properties similar to those of the classical Jacobi polynomials.

\begin{proposition}\label{Jac-Mul-35}
Let $\{ \tilde{P}_n^{(\alpha,\beta)}\}_{n\geq 0}$ be the sequence of multiplicative Jacobi polynomials. The following properties hold:
\begin{enumerate}
    \item $\frac{d^*}{dx}\left(\tilde{P}_n^{(\alpha,\beta)}(x)\right)=\left(\tilde{P}_{n-1}^{(\alpha+1,\beta+1)}(x)\right)^{\frac{1}{2}(n+\alpha+\beta+1)}$ and
    \begin{equation*}
        \frac{d}{dx}\left(\tilde{P}_n^{(\alpha,\beta)}(x)\right)=\frac{1}{2}(n+\alpha+\beta+1)P_{n-1}^{(\alpha+1,\beta+1)}(x)\tilde{P}_n^{(\alpha,\beta)}(x).
    \end{equation*}
\item Multiplicative Rodrigues' formula
\begin{equation*}
    \tilde{P_n}^{(\alpha,\beta)}(x)=\left[\frac{d^{*(n)}}{dx^n}\left( e^{(1-x^2)^n\omega(x)}\right)\right]^{\frac{1}{(-1)^nn!2^n\omega(x)}}, \ \ n\geq 0.
\end{equation*}
\item Three-term multiplicative recurrence relation 
\begin{equation*}
    \left(\tilde{P}_{n}^{(\alpha,\beta)}(x) \right)^{A(n,\alpha,\beta)}=\frac{\left(\tilde{P}_{n-1}^{(\alpha,\beta)}(x) \right)^{B(x,n,\alpha,\beta)}}{\left(\tilde{P}_{n-2}^{(\alpha,\beta)}(x) \right)^{C(n,\alpha,\beta)}}, \ \ n\geq 1,
\end{equation*}
with $\tilde{P}_{0}^{(\alpha,\beta)}(x)=e$ and $\tilde{P}_{-1}^{(\alpha,\beta)}(x)=1$, where $A(n,\alpha,\beta)$, $B(x,n,\alpha,\beta)$ and $C(n,\alpha,\beta)$ are as in \eqref{Jac-Mul-32}.
\item $\tilde{P_n}^{(\alpha,\beta)}(1)=e^{\binom{n+\alpha}{n}}$ and $\tilde{P_n}^{(\alpha,\beta)}(-x)=\left( \tilde{P_n}^{(\beta,\alpha)}(x)\right)^{(-1)^n}$.
\item $\tilde{P}_n^{(\alpha,\beta)}(x)$ is a multiplicative polynomial of degree $n$ whose leading coefficient is 
\begin{equation*}
    \tilde{k}_n\equiv\tilde{k}_n(\alpha,\beta)=e^{\frac{1}{2^n}\binom{2n+\alpha+\beta}{n}}, \ \ n\geq 0.
\end{equation*}
\end{enumerate}
\end{proposition}
\begin{proof}
We only prove the first two, since the others are immediate from the properties given in Propositions \ref{Jac-Mul-31} and \eqref{Jac-Mul-28}.
    \begin{enumerate}
        \item From \eqref{Jac-Mul-23-b} and \eqref{Jac-Mul-28} we have $\frac{d^*}{dx}\left(\tilde{P}_n^{(\alpha,\beta)}(x)\right)=e^{(P_n^{(\alpha,\beta)})'}=e^{\frac{1}{2}(n+\alpha+\beta+1)P_{n-1}^{(\alpha+1,\beta+1)}}$, this is the first property.
    \item Using Rodrigues' formula for classical Jacobi polynomials we have
    \begin{equation}\label{Jac-Mul-34}        \tilde{P}_n^{(\alpha,\beta)}(x)=\left( e^{[(1-x^2)^n\omega(x)]^{(n)}} \right)^{\frac{1}{(-1)^nn!2^n\omega(x)}}.
    \end{equation}
From \eqref{Jac-Mul-33}, with $f(x)=e$ and $\phi=(1-x^2)^n\omega(x)$, we get $e^{[(1-x^2)^n\omega(x)]^{(n)}}=\frac{d^{*(n)}}{dx^n}\left( e^{(1-x^2)^n\omega(x)}\right)$, substituting in \eqref{Jac-Mul-34} the result follows.
    \end{enumerate}
\end{proof}

Some particular cases, famous in classical calculus (see \cite{Chihara}), that have analogs in multiplicative calculus are the following: 
\begin{enumerate}
    \item The multiplicative Legendre polynomials ($\alpha=\beta=0$)
    \begin{equation*}
        \tilde{P}_n(x)=e^{P_n^{(0,0)}(x)}, \ \ n\geq 0.
    \end{equation*}
    \item The multiplicative Chebyshev polynomials of the first kind ($\alpha=\beta=-\frac{1}{2}$)
    \begin{equation*}
        \tilde{T}_n(x)=e^{2^{2n}\binom{2n}{n}^{-1}P_n^{(-\frac{1}{2},-\frac{1}{2})}(x)}, \ \ n\geq 0.
    \end{equation*}
    \item The multiplicative Chebyshev polynomials of the second kind ($\alpha=\beta=\frac{1}{2}$)
    \begin{equation*}
        \tilde{U}_n(x)=e^{2^{2n}\binom{2n+1}{n+1}^{-1}P_n^{(\frac{1}{2},\frac{1}{2})}(x)}, \ \ n\geq 0.
    \end{equation*}
    \item The multiplicative Gegenbauer (or Ultraspherical) polynomials ($\alpha=\beta\neq -\frac{1}{2}$)
    \begin{equation*}
        \tilde{P}_n^{(\alpha+\frac{1}{2})}(x)=e^{\binom{2\alpha}{\alpha}^{-1}\binom{n+2\alpha}{\alpha}P_n^{(\alpha,\alpha)}(x)}, \ \ n\geq 0.
    \end{equation*}
\end{enumerate}

\begin{example}
According to \eqref{Jac-Mul-28} the first six multiplicative Chebyshev polynomials of the first kind are 
\begin{multicols}{2}
\begin{figure}[H]
\centering
\includegraphics[width=7cm, height=3.2cm]{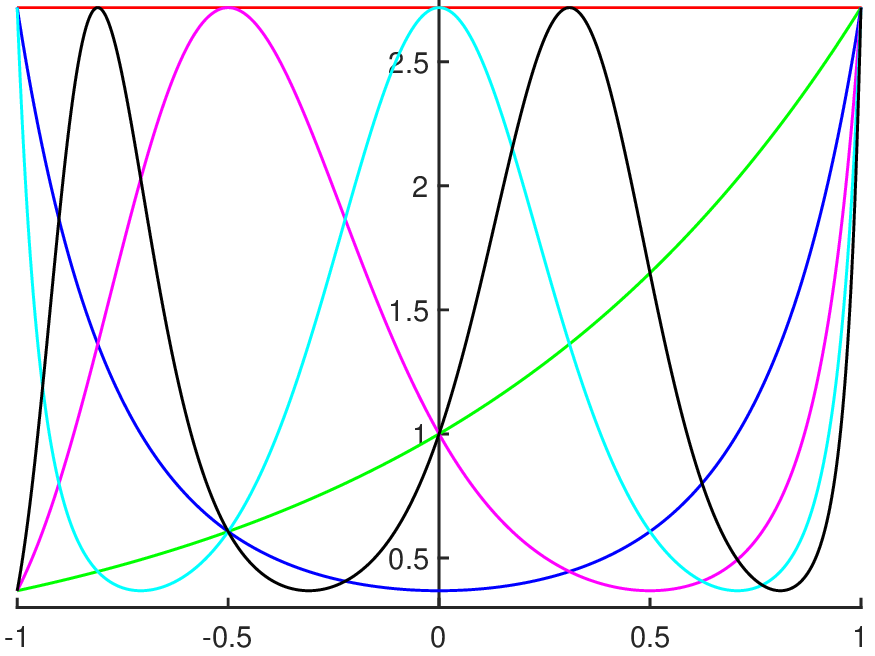}
\caption{ Plots of the first six Chebyshev multiplicative polynomials of the first kind on the interval $[-1, 1]$.}
\label{fig:figura3}
\end{figure}
%\columnbreak
\begin{equation*}
    \begin{split}
  \tilde{T}_0(x)&=e^1\  (\textcolor{red}{\rule{0.5cm}{0.5mm}}),  \\
  \tilde{T}_1(x)&=e^{x} \  (\textcolor{green}{\rule{0.5cm}{0.5mm}}),  \\
  \tilde{T}_2(x)&=e^{2x^2-1}\  (\textcolor{blue}{\rule{0.5cm}{0.5mm}}),\\
  \tilde{T}_3(x)&=e^{4x^3-3x}\  (\textcolor{magenta}{\rule{0.5cm}{0.5mm}}),\\
  \tilde{T}_4(x)&=e^{8x^4-8x^2+1}\  (\textcolor{cyan}{\rule{0.5cm}{0.5mm}}),\\
  \tilde{T}_5(x)&=e^{16x^5-20x^3+5x}\  (\textcolor{black}{\rule{0.5cm}{0.5mm}}).
    \end{split}
\end{equation*}
\end{multicols}
\noindent Notice that, for all $n$, it is possible to know the value on the boundary by using Proposition \ref{Jac-Mul-35}
\begin{equation*}
 \tilde{T}_n(1)=e^{2^{2n}\binom{2n}{n}^{-1}}\binom{n-\frac{1}{2}}{n}=e \mbox{ and } \tilde{T}_n(-1)=e^{(-1)^n}.
\end{equation*}
\end{example}

As in classical calculus, multiplicative Chebyshev polynomials of the first and second kind satisfy several properties stated next. We do not present the proof since they can be easily deduced from the properties of classical Chebyshev polynomials and the definition of their multiplicative analogs. The associated properties in the classical case can be found in \cite{Boy,Chihara}.
\begin{proposition}
    Let $\{ \tilde{P}_n\}_{n\geq 0}$, $\{ \tilde{T}_n\}_{n\geq 0}$, and $\{ \tilde{U}_n\}_{n\geq 0}$ be the multiplicative Legendre, Chebyshev of the first and second kind polynomials, respectively. They satisfy the following properties:
    \begin{enumerate}
    \item Orthogonality 
\begin{equation*}
    \begin{split}
    &\int_{-1}^1 \left(\tilde{P}_n(x)\odot \tilde{P}_m(x)\right)^{dx}=\exp\left(\frac{2}{2n+1}\delta_{n,m}\right),\\    
    &\int_{-1}^1 \left(\tilde{U}_n(x)\odot \tilde{U}_m(x)\right)^{\sqrt{1-x^2}dx}=\exp\left(\frac{\pi}{2}\delta_{n,m}\right),\\
    &\int_{-1}^1 \left(\tilde{T}_n(x)\odot \tilde{T}_m(x)\right)^{\frac{1}{\sqrt{1-x^2}}dx}= \left\{ \begin{array}{lcl} 1, & if & n \neq m,  \\ 
     e^{\pi}, & if & n=m=0,
    \\ e^{\frac{\pi}{2}}, & if & n=m. \end{array} \right.
    \end{split}
\end{equation*}
    \item Three-term multiplicative recurrence relation
        \begin{equation*}
    \tilde{P}_{0}(x)
    =e, \ \tilde{P}_{1}(x)=e^x, \mbox{ and }\tilde{P}_{n+1}(x)=\frac{( \tilde{P}_{n}(x))^{\frac{2n+1}{n+1}x}}{( \tilde{P}_{n-1}(x))^{\frac{n}{n+1}}}, 
    \end{equation*}
    \begin{equation*}
    \tilde{T}_{0}(x)
    =e, \ \tilde{T}_{1}(x)=e^x, \mbox{ and }\tilde{T}_{n+1}(x)=\frac{( \tilde{T}_{n}(x))^{2x}}{\tilde{T}_{n-1}(x)}, 
    \end{equation*}
    \begin{equation*}
    \tilde{U}_{0}(x)
    =e, \ \tilde{U}_{1}(x)=e^{2x}, \mbox{ and }\tilde{U}_{n+1}(x)=\frac{( \tilde{U}_{n}(x))^{2x}}{\tilde{U}_{n-1}(x)}.   
    \end{equation*}
    \item Multiplicative derivative, for $n\geq 1$ 
        \begin{equation*}
        \frac{d^*}{dx}\tilde{P}_n(x)=(\tilde{P}_{n-1}(x))^n, \ \frac{d^*}{dx}\tilde{T}_n(x)=(\tilde{U}_{n-1}(x))^n,  \mbox{ and } \   \frac{d^*}{dx}\tilde{U}_n(x)=
        \frac{(\tilde{T}_{n+1}(x))^{\frac{n+1}{x^2-1}}}{(\tilde{U}_{n}(x))^{\frac{x}{x^2-1}}}.
        \end{equation*}
        \item Multiplicative integral
        \begin{equation*}
            \int \left(\tilde{U}_n(x)\right)^{dx}=(\tilde{T}_{n+1}(x))^{\frac{1}{n+1}},\ \
            \int \left(\tilde{T}_n(x)\right)^{dx}=\frac{(\tilde{T}_{n+1}(x))^{\frac{n}{n^2-1}}}{(\tilde{T}_{n}(x))^{\frac{x}{n-1}}}.
        \end{equation*}
        \item Relations between Chebyshev multiplicative polynomials of the first and second kind
\begin{equation*}
    \tilde{T}_n(x)=\frac{\tilde{U}_{n}(x)}{\tilde{U}_{n-2}(x)}\odot e^{\frac{1}{2}}, \ \tilde{T}_{n+1}(x)=\frac{(\tilde{T}_{n}(x) )^x}{(\tilde{U}_{n-1}(x) )^{1-x^2}},  \mbox{ and } \ \tilde{T}_n(x)=\frac{\tilde{U}_{n}(x)}{(\tilde{U}_{n-1}(x))^x}.
\end{equation*}
    \end{enumerate}
\end{proposition}

%===============================================================================================
%==========================================================================================
%==========================================================================================
%==========================================================================================
\section{Approximation of positive Functions using multiplicative Jacobi-Fourier series}

In this section, we first show that positive functions defined on $[-1,1]$ can be represented by a multiplicative Jacobi-Fourier series. Later on, some examples are presented in which the use multiplicative polynomials instead of classical polynomials is more convenient for approximation purposes.

\begin{theorem}\label{Jac-Mul-54}
Let $\mathbf{L}^2_{*}\left([-1,1],\omega\right)$ be a weighted multiplicative space with respect to the weight $\omega(x)=(1-x)^{\alpha}(1+x)^{\beta}$ as in \eqref{Jac-Mul-201}. If $f\in \mathbf{L}_{*}^{2}\left([-1,1],\omega\right)$, then $f$ can be expressed uniquely as a multiplicative Jacobi-Fourier series of the form
\begin{equation}\label{Jac-Mul-47}
f(x)= \prod_{n=0}^{\infty}e^{f_n} \odot \tilde{P_n}^{(\alpha,\beta)}(x)=\prod_{n=0}^{\infty}e^{f_nP_n^{(\alpha,\beta)}(x)}, \ \ x\in[-1,1],
\end{equation}
where the number $f_n$ is said to be the $n$-th multiplicative Jacobi-Fourier coefficient and 
\begin{equation}\label{Jac-Mul-44}
\begin{split}
    f_n&=\frac{\ln \left( \langle f,\tilde{P}_n^{(\alpha,\beta)}\rangle_{*,\omega} \right)}{\ln\left( \|\tilde{P}_n^{(\alpha,\beta)}\|_{*,\omega}^2 \right)}\\
    &=  \frac{2n+\alpha+\beta+1}{2^{\alpha+\beta+1}}\frac{n!\Gamma(n+\alpha+\beta+1)}{\Gamma(n+\alpha+1)\Gamma(n+\beta+1)} \int_{-1}^1\ln(f(x))P_{n}^{(\alpha,\beta)}(x)\omega(x)dx.
\end{split}
\end{equation}
Moreover, the multiplicative Jacobi-Fourier series converges to $f$ in the following way
\begin{equation}\label{Jac-Mul-48}
\lim_{N\to \infty} \left\| \frac{f}{\prod_{n=0}^{N}e^{f_n}\odot \tilde{P}_n^{(\alpha,\beta)}}  \right\|_{*,\omega}=1.
\end{equation}
\end{theorem}
\begin{proof}
    Since $f\in \mathbf{L}_{*}^{2}\left([-1,1],\omega\right)$, from Lemma \ref{Jac-Mul-43} we have $\ln (f)\in \mathbf{L}^{2}\left([-1,1],\omega\right)$. 
    Since $\{P_n^{(\alpha,\beta)} \}_{n\geq 0} $ is a complete (total) basis for the weighted Hilbert space $\mathbf{L}^{2}\left([-1,1],\omega\right)$, then $\ln f$ admits the following expansion 
    \begin{equation}\label{Jac-Mul-46}
    \ln (f(x))=\sum_{n=0}^{\infty}f_nP_n^{(\alpha,\beta)}(x), \ \ x\in [-1,1],
    \end{equation}
    where $f_n=\frac{  \langle \ln(f),P_n^{(\alpha,\beta)}\rangle_{\omega} }{ \|P_n^{(\alpha,\beta)}\|_{\omega}^2}$. From these coefficients we can deduce the multiplicative Jacobi-Fourier coefficients as in \eqref{Jac-Mul-44}, by using \eqref{Jac-Mul-28-b}, \eqref{Jac-Mul-45} and \eqref{Jac-Mul-42}. In a similar way, \eqref{Jac-Mul-47} is immediate from \eqref{Jac-Mul-46}.

Moreover, since the Jacobi-Fourier series converges to $\ln f$ in the weighted norm, then from \eqref{Jac-Mul-28-b} and \eqref{Jac-Mul-42} we have
\begin{equation*}
\lim_{N \to \infty} \left( \ln \left\| \frac{f}{\prod_{n=0}^{N}e^{f_n}\odot \tilde{P}_n^{(\alpha,\beta)}}  \right\|_{*,\omega}\right)=\lim_{N \to \infty}\left\| \ln f-\sum_{n=0}^{N}f_nP_n^{(\alpha,\beta)} \right\|_{\omega}=0,
\end{equation*} 
 which is equivalent to \eqref{Jac-Mul-48}.
\end{proof}

Notice that $\{ P_n^{(\alpha,\beta)}\}_{n\geq 0}$ constitutes a basis in the linear space of real polynomials on $[-1,1]$, and thus any such polynomial of degree $N$ can be (uniquely) expressed as $\pi(x)=\sum_{n=0}^N f_n P_n^{(\alpha,\beta)}(x)$. As a consequence, we have the following result. 
\begin{corollary}\label{Jac-Mul-51}
    Let $r(x)$ be a real polynomial of degree $N$ on $[-1,1]$ and be $a\in \mathbb{R}_e$. If $f(x)=a^{r(x)}=e^{r(x)\ln a}$, then
    \begin{equation*}
        f(x)=\prod_{n=0}^{N}e^{f_n} \odot \tilde{P_n}^{(\alpha,\beta)}(x)=\prod_{n=0}^{N}e^{f_nP_n^{(\alpha,\beta)}(x)}, \ \ x\in[-1,1],
    \end{equation*}
where
\begin{equation}\label{Jac-Mul-57}
 f_n=\frac{2n+\alpha+\beta+1}{2^{\alpha+\beta+1}}\frac{n!\Gamma(n+\alpha+\beta+1)}{\Gamma(n+\alpha+1)\Gamma(n+\beta+1)} \int_{-1}^1\ln(a)r(x)P_{n}^{(\alpha,\beta)}(x)\omega(x)dx.   
\end{equation}
\end{corollary}

In fact, when considering a product of a polynomial function and an exponential function whose exponent is a polynomial, it is possible to combine classical series with multiplicative series in order to obtain a series representation with a finite number of terms. In general, for a product of functions, classical and multiplicative series can be combined to obtain a better approximation of the product function. 
\begin{corollary}\label{Jac-Mul-200}
        Let $r(x)$ and $\pi(x)$ be real polynomials of degree $N$ and $M$ on $[-1,1]$, respectively, and be $a\in \mathbb{R}_e$. If $f(x)=\pi(x)a^{r(x)}=\pi(x)e^{r(x)\ln a}$, then
    \begin{equation*}
        f(x)=\left(\sum_{n=0}^{M}c_nP_n^{(\alpha,\beta)}(x)\right)\left(\prod_{n=0}^{N}e^{f_nP_n^{(\alpha,\beta)}(x)}\right), \ \ x\in[-1,1],
    \end{equation*}
where $f_n$ is as in \eqref{Jac-Mul-57} and
\begin{equation}\label{Jac-Mul-58}
    c_n=\frac{2n+\alpha+\beta+1}{2^{\alpha+\beta+1}}\frac{n!\Gamma(n+\alpha+\beta+1)}{\Gamma(n+\alpha+1)\Gamma(n+\beta+1)} \int_{-1}^1\pi(x)P_{n}^{(\alpha,\beta)}(x)\omega(x)dx.
\end{equation}
In general, if $r(x)$ and $\pi(x)$ are arbitrary functions with $r(x)$ positive on $[-1,1]$, then
\begin{equation}\label{Jac-Mul-59}
f(x)=\pi(x)r(x)=\pi(x)e^{\ln(r(x))}=\left(\sum_{n=0}^{\infty}c_nP_n^{(\alpha,\beta)}(x)\right)\left(\prod_{n=0}^{\infty}e^{f_nP_n^{(\alpha,\beta)}(x)}\right), \ \ x\in[-1,1],
\end{equation}
where $c_n$ is as in \eqref{Jac-Mul-58} and $f_n$ is as in \eqref{Jac-Mul-44} with $f(x)=r(x)$.
\end{corollary}

Notice that in \eqref{Jac-Mul-59} it can happen that $\pi(x)$ or $\ln(r(x))$ are polynomials. In such a case, the expression can be reduced to a finite sum or product.

Now, we present some examples to illustrate that the use of multiplicative Jacobi-Fourier series to approximate positive functions can be more convenient than the use of classical series.
\begin{example}
    Consider the Gaussian function $f(x)=e^{-100\left(x-\frac{1}{5}\right)^2}$, $x\in[-1,1]$. This function has a fast decay (exponential decay) to zero and, as a consequence, if it is approximated in a classical way by using partial sums of Jacobi-Fourier expansions of the form
\begin{equation}\label{Jac-Mul-40} e^{-100\left(x-\frac{1}{5}\right)^2} \approx \sum_{n=0}^N \frac{\langle e^{-100\left(x-\frac{1}{5}\right)^2},P_n^{(\alpha,\beta)} \rangle_{\omega}}{\|P_n^{(\alpha,\beta)} \|_{\omega}} P_{n}^{(\alpha,\beta)}(x)=\Sigma(N,\alpha,\beta),
\end{equation}
then a considerable error is obtained, mainly around the symmetry axis $x=0.2$ and in the boundary points $x=\pm1$, as showed in Figures \ref{fig:Legendre-2} and \ref{fig:Legendre-3}, where the plots of $f$ and $\Sigma(20,0,0)$ are presented using classical Legendre polynomials with $N=20$
\begin{multicols}{2}
\begin{figure}[H]
\centering
\includegraphics[width=6cm, height=3.5cm]{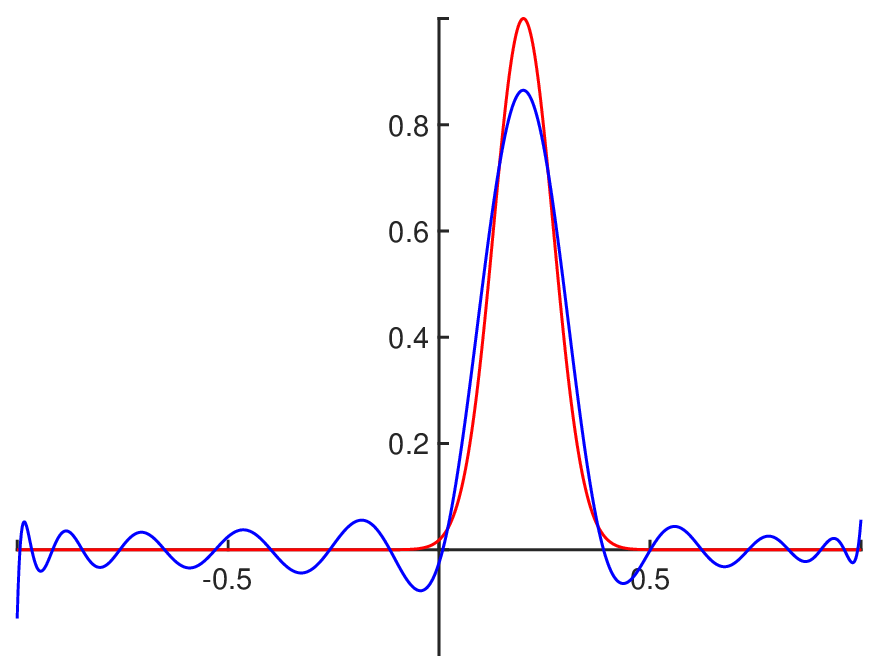}
\caption{$f(x)$ (\textcolor{red}{\rule{0.5cm}{0.5mm}}) and the associated Legendre approximation $\Sigma(20,0,0)$ (\textcolor{blue}{\rule{0.5cm}{0.5mm}}) on $[-1,1]$.}
\label{fig:Legendre-2}
\end{figure}
\columnbreak
\begin{figure}[H]
\centering
\includegraphics[width=6cm, height=3.5cm]{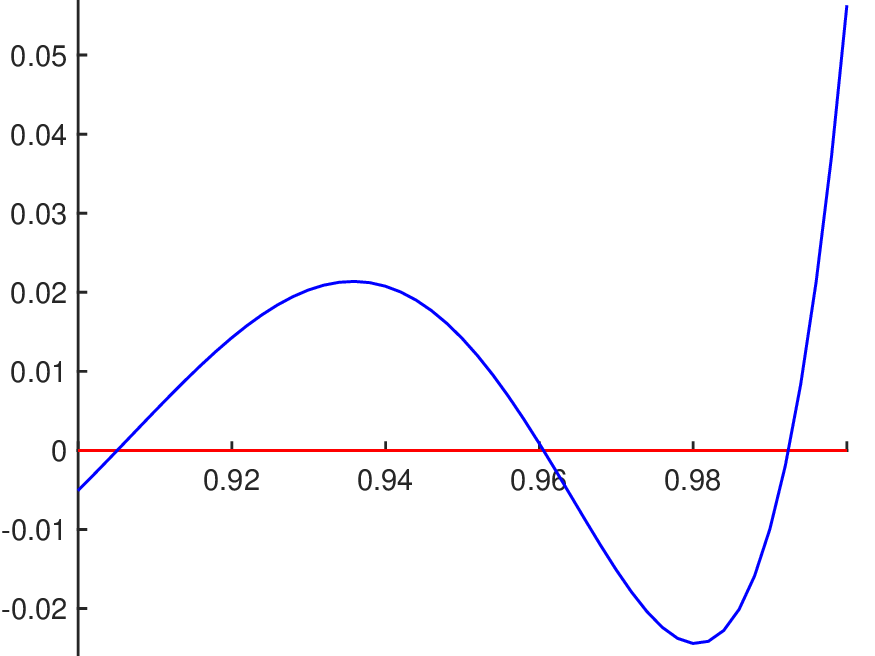}
\caption{Close up of $f$ and $\Sigma(20,0,0)$ on $[0.9, 1 ]$.}
\label{fig:Legendre-3}
\end{figure}
\end{multicols}
Furthermore, if $\alpha,\beta$ are positive real numbers, then the approximation error increases even more. A better approximation is presented in \cite{Iserles,Jesus2} by using Fourier series similar to \eqref{Jac-Mul-40}, but using Sobolev orthogonal polynomials, which have in addition a coherence property (for more details in Sobolev polynomials and the coherence property, we refer the reader to \cite{Marcellan}). However, even with Sobolev orthogonality, the approximation at the center ($x=0.2$ in this example) and the endpoints is rather slow (see Figure 3 from \cite{Jesus2}). On the other hand, if the multiplicative Jacobi polynomials are used as basis, faster results are achieved. For instance, using he multiplicative Legendre polynomials ($\alpha=\beta=0$, $\tilde{P}_n^{(\alpha,\beta)}(x)=\tilde{P}_n(x)$) and Corollary \eqref{Jac-Mul-51}, we get
\begin{equation*}
e^{-100\left(x-\frac{1}{5}\right)^2}=\prod_{n=0}^2e^{f_nP_n(x)}, \ \ f_n=\frac{2n+1}{2}\int_{-1}^1\left( -100 \left( x-\frac{1}{5}\right)^2 \right)P_n(x)dx,
\end{equation*}
%\begin{multicols}{2}
%\noindent with $P_0(x)=1$, $P_1(x)=x$, and $P_2(x)=\frac{1}{2}(3x^2-1)$ are the classical Legendre polynomials. Thus
\begin{equation*}
f_0=\frac{1}{2}\left(-\frac{224}{3} \right), \quad f_1=\frac{3}{2}\left(\frac{80}{3} \right), \quad f_2=\frac{5}{2}\left(-\frac{80}{3} \right).
%\begin{split}
%    f_0&=\frac{1}{2}\left(-\frac{224}{3} \right), \\  
%    f_1&=\frac{3}{2}\left(\frac{80}{3} \right),\\
%    f_2&=\frac{5}{2}\left(-\frac{80}{3} \right).
%\end{split}
\end{equation*}
%\columnbreak
%\begin{figure}[H]
%\centering
%\includegraphics[width=7.3cm, height=3.5cm]{Figuras/Gauss_Multi.eps}
%\caption{Plots $f(x)$ (\textcolor{red}{\rule{0.5cm}{0.5mm}}), 
%$e^{f_0P_0(x)}$ (\textcolor{yellow}{\rule{0.5cm}{0.5mm}}),
%$\prod_{n=0}^1e^{f_nP_n(x)}$ (\textcolor{green}{\rule{0.5cm}{0.5mm}}), and
%$\prod_{n=0}^2e^{f_nP_n(x)}$ (\textcolor{blue}{\rule{0.5cm}{0.5mm}})
%on $[-1, 1]$.}
%\label{fig:Legendre-15}
%\end{figure}
%\end{multicols}

\begin{figure}[H]
\centering
\includegraphics[width=8.3cm, height=4.5cm]{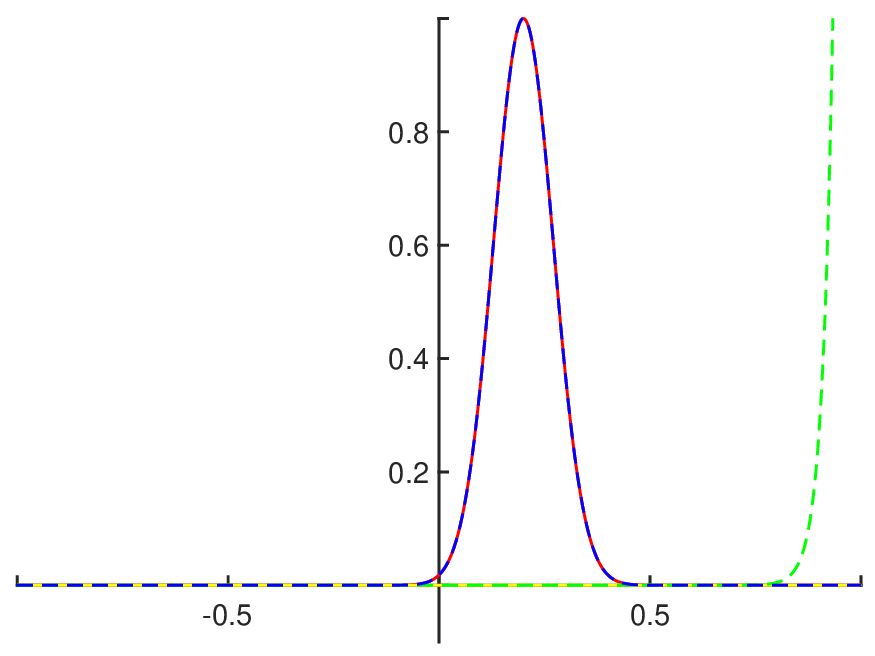}
\caption{Plots $f(x)$ (\textcolor{red}{\rule{0.5cm}{0.5mm}}), 
$e^{f_0P_0(x)}$ (\textcolor{yellow}{\rule{0.5cm}{0.5mm}}),
$\prod_{n=0}^1e^{f_nP_n(x)}$ (\textcolor{green}{\rule{0.5cm}{0.5mm}}), and
$\prod_{n=0}^2e^{f_nP_n(x)}$ (\textcolor{blue}{\rule{0.5cm}{0.5mm}})
on $[-1, 1]$.}
\label{fig:Legendre-15}
\end{figure}

\noindent Thus, the approximation process using multiplicative series is more efficient if the function has exponential nature, as can be seen in Figure \ref{fig:Legendre-15}.
\end{example}

\begin{example}
Consider the function $f(x)=\frac{1}{1+25x^2}$, $x\in[-1,1]$, which has the particularity that, when approximated by  Lagrange polynomials by using a equidistant partition (or equidistant nodes) of the interval $[-1,1]$, has a large oscillation near the endpoints  (see \cite{Burden}). This is known as Runge phenomena (see Figure \ref{fig:Legendre-10}).
\begin{figure}[H]
\centering
\includegraphics[width=8.5cm, height=3.8cm]{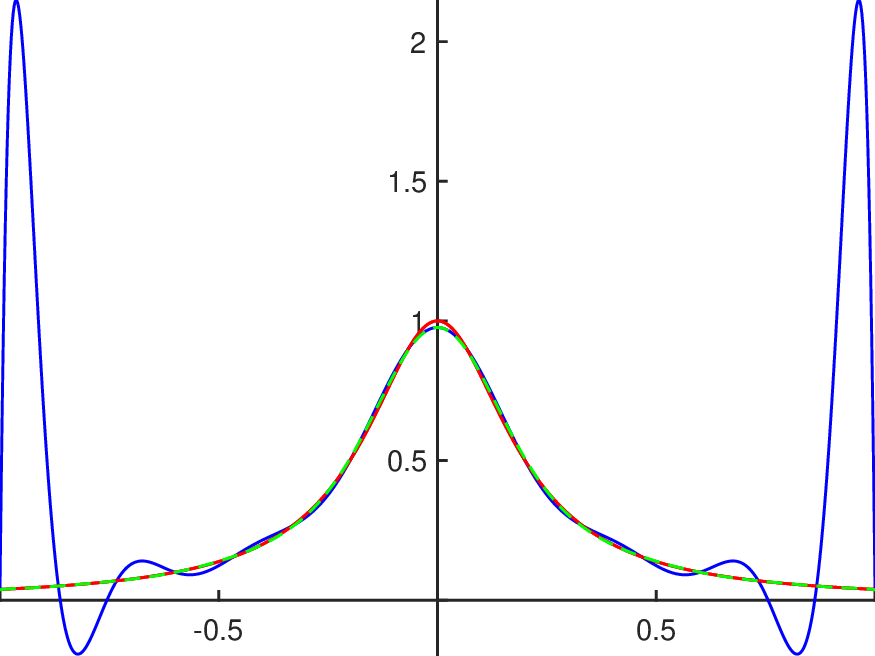}
\caption{Plots $f(x)=\frac{1}{1+25x^2}$ (\textcolor{red}{\rule{0.5cm}{0.5mm}}), Lagrange interpolation polynomial of degree $15$ (\textcolor{blue}{\rule{0.5cm}{0.5mm}}) and $\Pi(15)$ (\textcolor{green}{\rule{0.5cm}{0.5mm}}) on $[-1, 1]$.}
\label{fig:Legendre-10}
\end{figure}
Several methods have been developed to correct this phenomena. For instance, when choosing the interpolation nodes as the zeros of classical Chebyshev polynomials of the first kind, which are not equidistant and are more concentrated near the endpoints, the approximation can be improved \cite{Boy}. Also, in some cases, approximating the function by orthogonal functions that are not polynomials yields better results \cite{Boy}. 

Another way to mitigate these oscillations in positive functions with rapid growth or decay is the use of multiplicative polynomials, as shown below. For instance, taking $\alpha=\beta=0$, the function $f$ can be approximated by Legendre and multiplicative Legendre series in the following way
\begin{equation*}
  \begin{split}
     \frac{1}{1+25x^2} & \approx \sum_{n=0}^{N}\frac{2n+1}{2}\left(\int_{-1}^{1} \frac{1}{1+25x^2}P_n(x)dx \right) P_n(x)=\Sigma(N), \\
       \frac{1}{1+25x^2} & \approx \prod_{n=0}^{N}\exp\left(\frac{2n+1}{2} \left(\int_{-1}^{1} \ln\left(\frac{1}{1+25x^2}\right) P_n(x) dx \right) P_n(x) \right)=\Pi(N).
  \end{split}
\end{equation*}
Figure \ref{fig:Legendre-10} shows that $\Pi(15)$ removes the Runge phenomena in the approximation. Moreover, Figure \ref{fig:Legendre-40} shows that $\Pi(15)$ approximates $f$ uniformly better than $\Sigma(15)$. In fact,  $\Sigma(15)$ does not remove the oscillations in the boundary points as well as $\Pi(15)$, as can be seen in Figure \eqref{fig:Legendre-4}.
\end{example}

The previous examples show that approximating positive functions using multiplicative Jacobi series can be more efficient that using classical Jacobi series. However, this is not necessarily the case for all positive functions. Depending of the behavior of the function, the use of classical Jacobi series can be more effective. For instance, if $f(x)=\pi(x)=x+2$, which is positive on $[-1,1]$, the Jacobi multiplicative series \eqref{Jac-Mul-47} would have an infinite number of terms, whereas the Jacobi-Fourier classical series would simply be a linear combination of polynomials of degree $0$ and $1$. Thus, the use of a multiplicative series here would not be appropriated.

\begin{multicols}{2}
\begin{figure}[H]
\centering
\includegraphics[width=6.0cm, height=3.5cm]{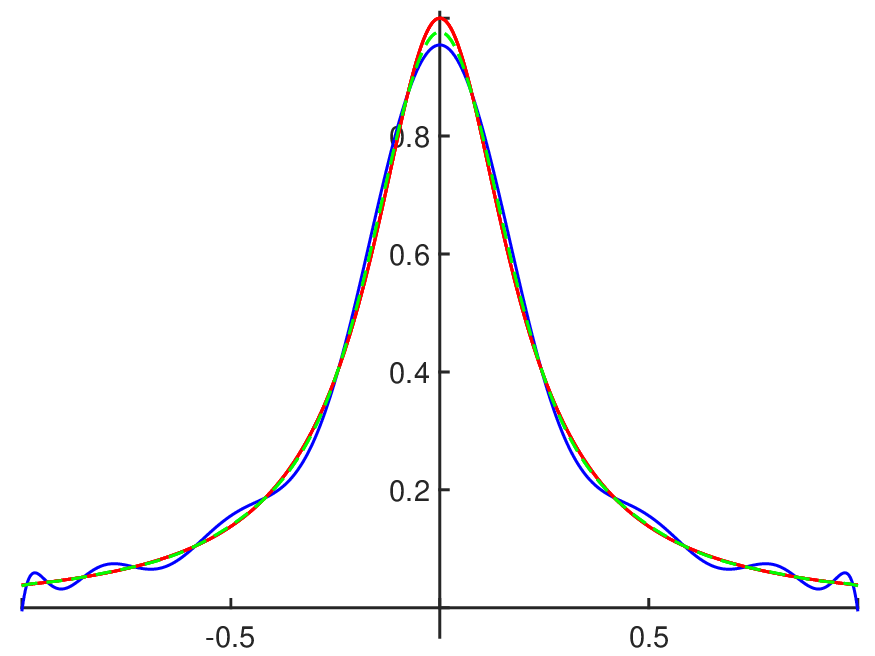}
\caption{Plots $f(x)$ (\textcolor{red}{\rule{0.5cm}{0.5mm}}), $\Sigma(15)$ (\textcolor{blue}{\rule{0.5cm}{0.5mm}}) and $\Pi(15)$ (\textcolor{green}{\rule{0.5cm}{0.5mm}}) on $[-1, 1]$.}
\label{fig:Legendre-40}
\end{figure}
\columnbreak
\begin{figure}[H]
\centering
\includegraphics[width=6.0cm, height=3.5cm]{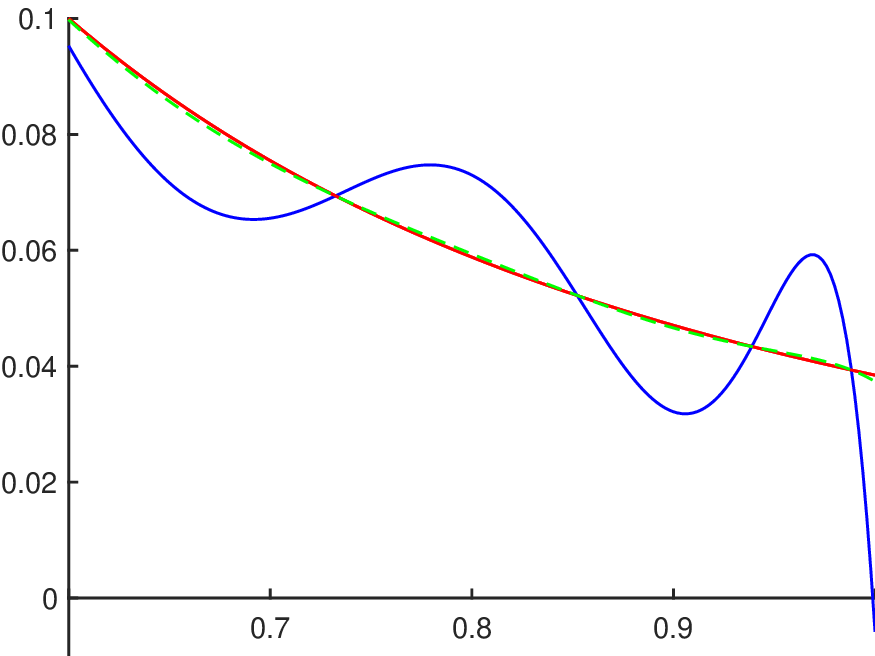}
\caption{Plots $f(x)$ (\textcolor{red}{\rule{0.5cm}{0.5mm}}), $\Sigma(15)$ (\textcolor{blue}{\rule{0.5cm}{0.5mm}}) and $\Pi(15)$ (\textcolor{green}{\rule{0.5cm}{0.5mm}}) on $[0.6, 1]$.}
\label{fig:Legendre-4}
\end{figure}
\end{multicols}

Now, we present some cases in which either classical or multiplicative approximation methods might be recommended to achieve more efficient results. Note that these conclusions are based on several numerical experiments and are presented without formal proof. As with any approximation technique, it is crucial to understand the behavior of the function to select the most appropriate method. Applying an approximation method without knowledge of the function's behavior can lead to unexpected results and significant errors in practical applications. To address this, we present the following definitions.

\begin{definition}
    A function $f:[a,b]\rightarrow \mathbb{R}$ is called:
\begin{enumerate}
    \item Of polynomial order on $[a,b]$ if there exists an non-negative integer $n$ and a positive constant $k$ such that
    \begin{equation*}%\label{Jac-Mul-53}
        |f(x)|<k|x|^n, \ \mbox{ for all } x\in[a,b].
    \end{equation*}
    This is commonly denoted by  $f\sim\mathcal{O}(x^n)$.
    \item Of exponential order on $[a,b]$ if there exists a positive constants $k$ and $m$ such that 
    \begin{equation*}
        |f(x)|<ke^{m|x|}, \ \mbox{ for all } x\in[a,b].
    \end{equation*}
    They are denoted by $f\sim\mathcal{O}(e^{m|x|})$.
\end{enumerate}
\end{definition}

%\textcolor{blue}{este tipo de funciones son de gran importacion en el estudio de la Laplace transform}

Functions of polynomial order grow slowly and are dominated by the term $x^n$, which implies that 
$f$ does not grow faster than a polynomial of degree $n$. As a consequence, Jacobi orthogonal polynomials and their associated Fourier series are powerful tools for approximating these kinds of functions because they take advantage of their structure (see \cite{Boy,Burden}). On the other hand, functions of exponential order grow faster than functions of polynomial order and exhibit more rapid oscillation. This results in a slower approximation process with Jacobi-Fourier series, requiring polynomials of higher degree. In contrast, multiplicative Jacobi polynomials and their associated multiplicative Jacobi series offer a better approximation because they have an exponential form. However, determining whether a function has polynomial or exponential order can sometimes be complicated, as a given function can exhibit characteristics of both polynomial and exponential order, as demonstrated by the following result. For such functions, deciding which approximation yields better results is, in general, an open problem.

\begin{lemma}
    Let $\pi,r:[a,b]\rightarrow \mathbb{R}$ be real functions and $f(x)=\pi(x)r(x)$. 
    \begin{enumerate}
        \item If $\pi\sim\mathcal{O}(x^n)$ and $r$ is bounded, then $f\sim\mathcal{O}(x^n)$.
        \item If $r\sim\mathcal{O}(e^{m|x|})$ and $\pi$ is bounded, then $f\sim\mathcal{O}(e^{m|x|})$.
    \end{enumerate}
\end{lemma}
\begin{proof}
  For the first statement, since $r(x)$ is bounded there exists a constant $k\in \mathbb{R}_e$ such that $|r(x)|<k$. Therefore
    \begin{equation*}%\label{Jac-Mul-55}
        |f(x)|=|\pi(x)r(x)|<k|\pi(x)|, \ \ x\in[a,b],
    \end{equation*}
    as $\pi\sim\mathcal{O}(x^n)$ the result is immediate. The second statement can be proven in similar way.
\end{proof}

The following is an example of a function satisfying the hypotheses of the previous lemma, i.e. the function has exponential and polynomial order. In this case, it is convenient to use Corollary \eqref{Jac-Mul-200}, since $f$ is the product of a polynomial and an exponential function. 

\begin{example}
Consider the function $f(x)=(2x^2+x^4)2^{\cos(10x)}$, $x\in[-1,1]$. In this case we have $\pi(x)=(2x^2+x^4)$ and $r(x)=2^{\cos(10x)}$. Notice that $f$ has both polynomial and exponential growth, since
\begin{equation*}
    |f(x)|\leq 2(2x^2+x^4)\ \mbox{and} \
    |f(x)|\leq 3e^{\ln2}e^{\ln(2)|x| }, \ \ \mbox{for all } \ x\in[-1,1].
\end{equation*}
We will use both approximations by using Chebyshev polynomials of the first kind ($\alpha=\beta=-\frac{1}{2}$): 
\begin{equation*}
\begin{split}
    (2x^2+x^4)2^{\cos(10x)} & \approx \frac{1}{\pi}\int_{-1}^1\frac{f(x)dx}{\sqrt{1-x^2}}+\frac{2}{\pi}\sum_{n=1}^{N}\left(\int_{-1}^1\frac{f(x)T_n(x)dx}{\sqrt{1-x^2}}\right) T_n(x)=\Sigma(N), \\
    (2x^2+x^4)2^{\cos(10x)} & \approx \exp \left(\frac{1}{\pi}\int_{-1}^{1}\frac{\ln f(x)dx}{\sqrt{1-x^2}} \right)\prod_{n=1}^{N}\exp \left(\left(\frac{2}{\pi}\int_{-1}^{1}\frac{\ln (f(x))T_n(x)dx}{\sqrt{1-x^2}} \right)T_n(x)\right) =\Pi(N).
\end{split}
\end{equation*}
The results are shown in Figures \ref{fig:Legendre-50} and \ref{fig:Legendre-51}. Notice that $\Sigma(15)$ and $\Sigma(25)$ present a better approximation than $\Pi(15)$ and $\Pi(25)$, respectively. However, convergence is slow in both cases, since the function is the product of a polynomial function and an exponential function. 
\begin{multicols}{2}
\begin{figure}[H]
\centering
\includegraphics[width=6.0cm, height=3.5cm]{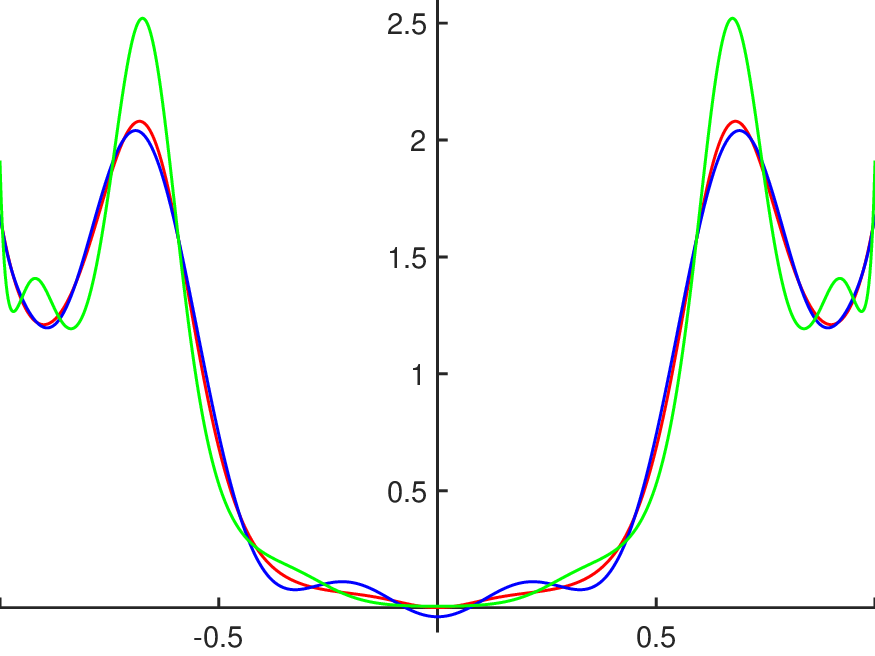}
\caption{Plots $f(x)$ (\textcolor{red}{\rule{0.5cm}{0.5mm}}), $\Sigma(15)$ (\textcolor{blue}{\rule{0.5cm}{0.5mm}}) and $\Pi(15)$ (\textcolor{green}{\rule{0.5cm}{0.5mm}}) on $[-1, 1]$.}
\label{fig:Legendre-50}
\end{figure}
\columnbreak
\begin{figure}[H]
\centering
\includegraphics[width=6.0cm, height=3.5cm]{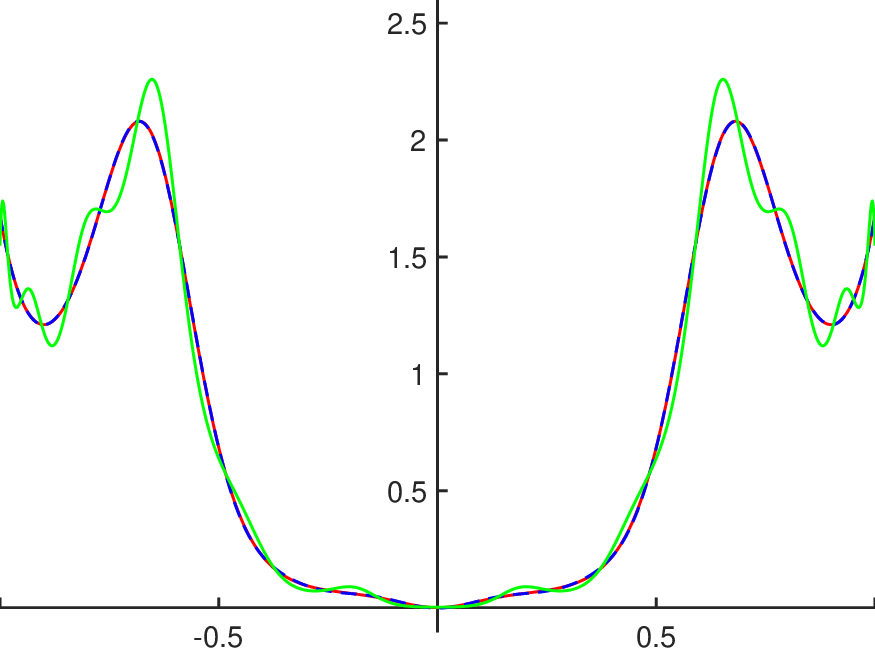}
\caption{Plots $f(x)$ (\textcolor{red}{\rule{0.5cm}{0.5mm}}), $\Sigma(25)$ (\textcolor{blue}{\rule{0.5cm}{0.5mm}}) and $\Pi(25)$ (\textcolor{green}{\rule{0.5cm}{0.5mm}}) on $[-1, 1]$.}
\label{fig:Legendre-51}
\end{figure}
\end{multicols}

On the other hand, since $f$ is the product of $\pi(x)=2x^2+x^4$ and $r(x)=2^{\cos(10x)}$, then we can approximate $f$ using \eqref{Jac-Mul-59}, i.e.
\begin{equation*}
    f(x)\approx (2x^2+x^4)\left( \prod_{n=0}^{N}e^{f_n T_n(x)} \right)=\textsc{P}\Pi(N),
\end{equation*}
where $f_0=\frac{\ln 2}{\pi}\int_{-1}^{1}\cos(10x)\frac{dx}{\sqrt{1-x^2}}$ and  $f_n=\frac{2\ln 2}{\pi}\int_{-1}^{1}\cos(10x)T_n(x)\frac{dx}{\sqrt{1-x^2}}$. That is, we only approximate $r$ by a multiplicative series, since for $N\geq 4$ we have $\pi=\Sigma(4)$. Figures \ref{fig:Legendre-54} and \ref{fig:Legendre-55} show the results of the approximation. It is easy to see that $\textsc{P}\Pi(15)$ approximate $f$ better than $\Sigma(15)$ and $\Pi(15)$, since the non-polynomial part is approximated by multiplicative polynomials.
\begin{multicols}{2}
\begin{figure}[H]
\centering
\includegraphics[width=6.0cm, height=3.5cm]{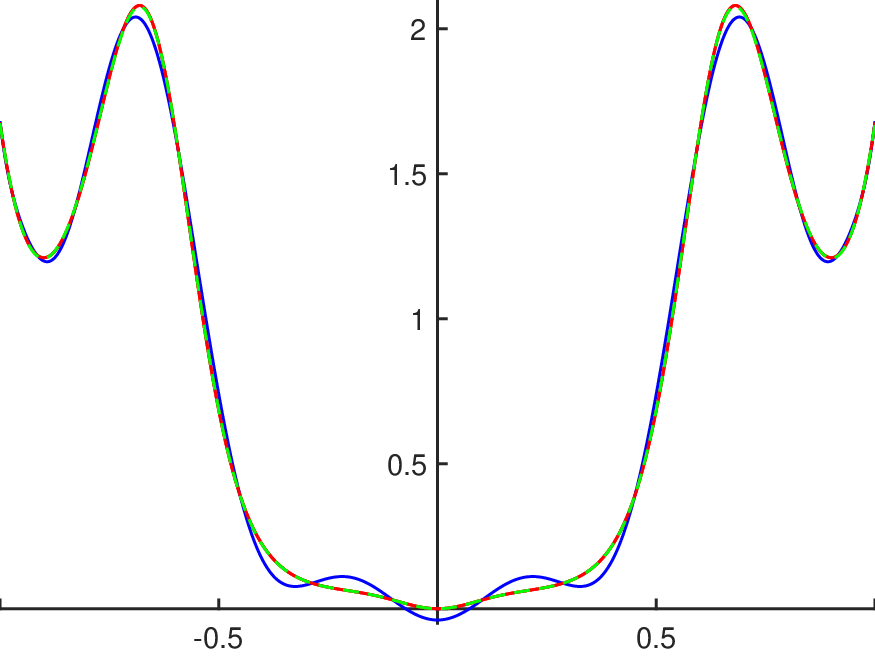}
\caption{Plots $f(x)$ (\textcolor{red}{\rule{0.5cm}{0.5mm}}), $\Sigma(15)$ (\textcolor{blue}{\rule{0.5cm}{0.5mm}}) and $\textsc{P}\Pi(15)$ (\textcolor{green}{\rule{0.5cm}{0.5mm}}) on $[-1, 1]$.}
\label{fig:Legendre-54}
\end{figure}
\columnbreak
\begin{figure}[H]
\centering
\includegraphics[width=6.0cm, height=3.5cm]{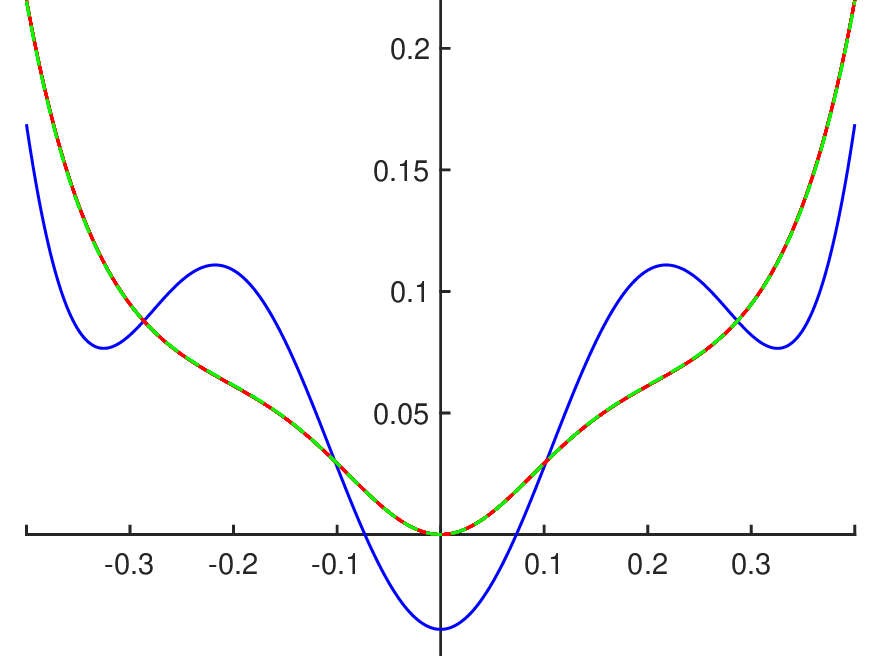}
\caption{Plots $f(x)$ (\textcolor{red}{\rule{0.5cm}{0.5mm}}), $\Sigma(15)$ (\textcolor{blue}{\rule{0.5cm}{0.5mm}}) and $\textsc{P}\Pi(15)$ (\textcolor{green}{\rule{0.5cm}{0.5mm}}) on $[-0.4, 0.4]$.}
\label{fig:Legendre-55}
\end{figure}
\end{multicols} 
In summary, if the behavior of the function is understood, it is possible to suggest a more precise approximation.
\end{example}

Finally, notice that the convergence of the series is with respect to the weighted norms \eqref{Jac-Mul-56} and \eqref{Jac-Mul-48}, and this means that there could be some points on $[-1,1]$ where the classical series approximation outperforms the multiplicative series approximation. In fact, uniform convergence is not guaranteed in any of the two approximations.

%===================================================================================================================================================
%===================================================================================================================================================
%===================================================================================================================================================

\section*{Acknowledgements}

The work of the second author was supported by Conahcyt Grant CBF2023-2024-625. The work of the first and third author has been supported by the Direcci\'on General de Investigaciones de la Universidad de los Llanos.

\end{document}